\documentclass[a4paper,12pt]{article}
\usepackage{tikz}
\usetikzlibrary{math,arrows,decorations.pathmorphing,decorations.pathreplacing,positioning,shapes.geometric,shapes.misc,decorations.markings,decorations.fractals,calc,patterns}
\def\hex#1#2{
  	\draw #1 #2 +(0:1) \foreach \a in {60,120,180,240,300} { -- +(\a:1) } -- cycle;
 	\draw[fill=white] #1 #2 \foreach \a in {0,120,240} { +(\a:1) circle (1.2mm) } ;
  	\draw[fill=black] #1 #2 \foreach \a in {60,180,300} { +(\a:1) circle (1.2mm) } ;
	}
\def\hexnode#1#2{
  	\draw #1 #2 +(0:1) \foreach \a in {60,120,180,240,300} { -- +(\a:1) } -- cycle;
 	\draw[fill=white] #1 #2 \foreach \a in {0,120,240} { +(\a:1) circle (2mm) } ;
  	\draw[fill=black] #1 #2 \foreach \a in {60,180,300} { +(\a:1) circle (2mm) } ;
	}
\def\hexp#1#2{
  	\draw #1 #2 +(0:0.75) \foreach \a in {60,120,180,240,300} { -- +(\a:0.75) } -- cycle;
 	\draw[fill=black] #1 #2 \foreach \a in {0,120,240} { +(\a:0.75) circle (1mm) } ;
  	\draw[fill=white] #1 #2 \foreach \a in {60,180,300} { +(\a:0.75) circle (1mm) } ;
	}
	
\def\hexpop#1#2{
  	\draw #1 #2 +(0:0.75) \foreach \a in {60,120,180,240,300} { -- +(\a:0.75) } -- cycle;
 	\draw[fill=white] #1 #2 \foreach \a in {0,120,240} { +(\a:0.75) circle (1mm) } ;
  	\draw[fill=black] #1 #2 \foreach \a in {60,180,300} { +(\a:0.75) circle (1mm) } ;
	}
	
\def\hexap#1#2{
  	\draw #1 #2 +(0:1) \foreach \a in {60,120,180,240,300} { -- +(\a:1) } -- cycle;
 	\draw[fill=black] #1 #2 \foreach \a in {0,60,120,180,240,300} { +(\a:1) circle (1mm) } ; 
  	\draw[fill=white] #1 #2 \foreach \a in {0,60,120,180,240,300} { +(\a-20:0.9) circle (1mm) } ;
 	}

\def\oct#1#2{
  	\draw[rotate=22.5] #1 #2 +(0:1) \foreach \a in {45,90,135,180,225,270,315} { -- +(\a:1) } -- cycle;
 	\draw[rotate=22.5,fill=black] #1 #2 \foreach \a in {0,90,180,270} { +(\a:1) circle (1.2mm) } ;
  	\draw[rotate=22.5,fill=white] #1 #2 \foreach \a in {45,135,225,315} { +(\a:1) circle (1.2mm) } ;
	}

\def\octop#1#2{
  	\draw[rotate=22.5] #1 #2 +(0:1) \foreach \a in {45,90,135,180,225,270,315} { -- +(\a:1) } -- cycle;
 	\draw[rotate=22.5,fill=white] #1 #2 \foreach \a in {0,90,180,270} { +(\a:1) circle (1.2mm) } ;
  	\draw[rotate=22.5,fill=black] #1 #2 \foreach \a in {45,135,225,315} { +(\a:1) circle (1.2mm) } ;
	}
	
\def\octopo#1#2{
  	\draw[thick,rotate=22.5] #1 #2 +(0:1) \foreach \a in {45,90,135,180,225,270,315} { -- +(\a:1) } -- cycle;
 	\draw[thick,rotate=22.5,fill=white] #1 #2 \foreach \a in {0,90,180,270} { +(\a:1) circle (0.8mm) } ;
  	\draw[thick,rotate=22.5,fill=black] #1 #2 \foreach \a in {45,135,225,315} { +(\a:1) circle (0.8mm) } ;
	}

\def\L{0.5412}
\def\Square#1#2{
  	\draw[rotate=45] #1 #2 +(0:\L) \foreach \a in {90,180,270} { -- +(\a:\L) } -- cycle;
 	\draw[rotate=45,fill=black] #1 #2 \foreach \a in {0,180} { +(\a:\L) circle (1.2mm) } ;
  	\draw[rotate=45,fill=white] #1 #2 \foreach \a in {90,270} { +(\a:\L) circle (1.2mm) } ;
	}
\def\RootOfThree{1.732050}
\def\RootOfSeven{2.6457513}

\def\AA{40.893}

\def\hh{0.86603}	
\newcommand{\Span}[1]{\left<#1\right>}
\newcommand{\dD}{\operatorname{D}}
\usepackage{typearea}
\typearea{15}

\usepackage{mymacros}
\usepackage{tikz-cd}

\newcommand{\dual}{\vee}
\newcommand{\height}{\operatorname{ht}}

\newcommand{\cEtilde}{\widetilde{\cE}}

\newcommand{\BD}{\operatorname{BD}}
\newcommand{\blue}{}
\newcommand{\red}{}

\title{Dimer models and group actions}
\author{Akira Ishii, \'Alvaro Nolla de Celis, Kazushi Ueda}
\date{}
\begin{document}

\maketitle

\begin{abstract}
We construct a consistent dimer model
having the same symmetry as its characteristic polygon.
This produces examples of non-commutative crepant resolutions
of non-toric non-quotient Gorenstein singularities in dimension 3.
\end{abstract}

%\tableofcontents

\section{Introduction}
 \label{sc:introduction}

A \emph{dimer model}
is a bicolored graph on a real 2-torus
encoding the information of a quiver with relations.
Dimer models are originally introduced in 1930s
\cite{Fowler-Rushbrooke}
as statistical mechanical models
of diatomic molecules,
which contain the Ising model as a special case.
See e.g. \cite{Baxter_ESMSM,Kenyon_IDM}
and references therein for this aspect of dimer models.
More recently,
string theorists has discovered the relation
between dimer models and toric Calabi-Yau 3-folds
\cite{Hanany-Kennaway_DMTD,Franco-Hanany-Vegh-Wecht-Kennaway_BDQGT,Franco-Hanany-Martelli-Sparks-Vegh-Wecht_GTTGBT,Hanany-Vegh},
and many works has been done
to explore the relation between dimer models and
various branches of mathematics,
such as Donaldson-Thomas theory
\cite{Szendroi_NCDT,Mozgovoy-Reineke},
Calabi-Yau algebras
\cite{Broomhead,Davison,Ishii-Ueda_CCDM,Bocklandt_CCDM,MR3010162},
volumes of toric Sasaki-Einstein 5-manifolds
\cite{Martelli-Sparks-Yau_GDAM,Butti-Zaffaroni_TGQGT,Butti-Zaffaroni_RTD,Kato_ZFSFT},
moduli spaces of quiver representations
\cite{Franco-Vegh_MSGTDM,Ishii-Ueda_08},
the McKay correspondence
\cite{Ishii-Ueda_DMSMC,MR3319545},
exceptional collections
\cite{Hanany-Herzog-Vegh_BTEC,Ishii-Ueda_DMEC},
and mirror symmetry
\cite{Feng-He-Kennaway-Vafa,Ueda-Yamazaki_NBTMQ,Ueda-Yamazaki_toricdP,Futaki-Ueda_A-infinity}.
%and integrable systems
%\cite{Goncharov-Kenyon_DCIS},

The \emph{characteristic polygon}
$\Delta$
of a dimer model $G$
is a convex lattice polygon
obtained from the dimer model
in a purely combinatorial way.
When $G$ satisfies a mild condition
called \emph{non-degeneracy},
the moduli space of representations of the quiver
associated with the dimer model
is a toric variety,
and
the convex hull
of the primitive generators of the one-dimensional cones
of the corresponding fan
coincides with $\Delta$
\cite{Ishii-Ueda_08}.
When $G$ satisfies a stronger condition
called \emph{consistency},
then the path algebra $\bC \Gamma$
of the associated quiver with relations $\Gamma$
is a \emph{non-commutative crepant resolution}
\cite{MR2077594}
of the affine toric variety $X_\Delta$
associated with $\Delta$.

%It is defined as a generating function
%of the \emph{height change},
%%obtained as the determinant of a matrix
%%called the \emph{Kasteleyn matrix},
%and determines the thermal behavior
%of the dimer model
%\cite{Kasteleyn,Kenyon-Okounkov-Sheffield}.
%The Newton polygon of the characteristic function
%is called the {\em characteristic polygon}.
%The characteristic polygon coincides
%with the convex hull
%of the primitive generators of the one-dimensional cones
%describing the toric variety
%associated with a dimer model

Let $H$ be a finite subgroup of $\GL(2,\bZ)$
acting naturally on the lattice
where the characteristic polygon $\Delta$ lives.
When $\Delta$ is invariant under this action,
then we can ask if the action can be `lifted'
to the dimer model $G$.
In this paper, we introduce the notion
of a \emph{symmetric dimer model}
with respect to the action of $H$,
and prove the following:

\begin{theorem} \label{th:main}
For any finite subgroup $H$ of $\GL(2,\bZ)$ and
any $H$-invariant lattice polygon $\Delta$,
there is a consistent dimer model $G$
which is symmetric with respect to the action of $H$ and has $\Delta$ as its characteristic polygon.
\end{theorem}

If a dimer model $G$ is symmetric
with respect to the action of a finite subgroup $H$ of $\GL(2,\bZ)$,
then $H$ acts on the associated quiver $\Gamma$ with relations.
There are associated actions of $H$ on $X_\Delta$ and $\bC\Gamma$
which are twisted as in \eqref{eq:twist_on_X} and \eqref{eq:twist_on_Gamma}
respectively.
Notice that if $H$ is a reflection group of order $2$ (see Remark \ref{rem:choice}), then the twist \eqref{eq:twist_on_X} depends on the choice of the origin in $\Delta$.
Moreover, the twist \eqref{eq:twist_on_Gamma} depends on the choice of an $H$-invariant
perfect matching corresponding to the origin.
With respect to these twisted actions, we prove:

\begin{theorem} \label{th:NCCR}
If a consistent dimer model $G$ is symmetric
with respect to the action of a finite subgroup $H$ of $\GL(2,\bZ)$,
then the crossed product algebra $\bC \Gamma \rtimes H$
is a non-commutative crepant resolution
of $X_\Delta/H$.
\end{theorem}

These two theorems imply the existence
of non-commutative crepant resolutions
of $X_\Delta/H$ which are not necessarily toric and
not necessarily quotient singularities.
This in turn implies the existence of crepant resolutions
by \cite{MR1893007,MR2057015,MR2077594},
which can also be shown directly
by first taking an $H$-invariant unimodular triangulation of $\Delta$
(which one can find by drawing line segments between the origin
and the corners of $\Delta$ to triangulate $\Delta$,
and then refining it to a unimodular triangulation)
to obtain an $H$-equivariant crepant resolution $Y$ of $X_\Delta$,
and then taking the Hilbert scheme $H\text{-Hilb}(Y)$.
It is an interesting problem to see if
every projective crepant resolution of $X_\Delta/H$ is obtained
as moduli of representations of $\bC \Gamma \rtimes H$
just as in \cite{MR2078369,MR3532120}.

%Let $Z_\Delta$ be the 2-dimensional toric Fano stack
%whose fan polytope (i.e., the convex hull of the primitive generators
%of one-dimensional cones in the stacky fan)
%is given by $\Delta$.
%Then
While the path algebra
of the quiver with relations
associated with a dimer model
is a 3-dimensional generalization
of that of the McKay quiver for a Kleinian singularity of type $A$,
the crossed product algebra $\bC \Gamma \rtimes H$
associated with a symmetric dimer model
is a 3-dimensional generalization
of that of type $D$,
and it is an interesting problem
to decide which constructions
on dimer models generalize to dimer models with group actions.
For example,
if we let $Z_\Delta$ denote the 2-dimensional toric Fano stack
whose fan polytope (i.e., the convex hull of the primitive generators
of one-dimensional cones in the stacky fan)
is given by $\Delta$,
then
one can show the existence
of a full strong exceptional collection
of vector bundles
on the stack quotient $\ld Z_\Delta /H \rd$
along the lines of
%\pref{th:NCCR} and
\cite[Theorem 1.1]{Ishii-Ueda_DMEC}.
%show that
%show the following:
%
%\begin{theorem} \label{th:EC}
%For any finite subgroup $H$ of $\GL(2,\bZ)$ and
%any $H$-invariant lattice polygon $\Delta$,

%whose total morphism algebra can be described explicitly.
%\end{theorem}
%
%The full strong exceptional collection in \pref{th:EC}
%and its total morphism algebra can be described explicitly
%in terms of the quiver with relations
%associated with an $H$-symmetric dimer model
%and a perfect matching on it.

This paper is organized as follows:
In Section \ref{sc:preliminaries},
we briefly recall basic definitions and results on dimer models.
More details can be found,
e.g., in \cite{MR2423955,MR3509904}
or references cited.
In Section \ref{sc:action},
we introduce the notion of a symmetric dimer model $G$
with respect to a finite subgroup $H$ of $\GL(2,\bZ)$
acting on the real 2-torus,
and discuss a quiver description
of the crossed product algebra $H \ltimes \bC \Gamma$
with the path algebra $\bC \Gamma$
of the associated quiver with relations.
After recalling the classification
of finite subgroups of $\GL(2,\bZ)$ in Section \ref{sc:subgroups},
we give an outline of the proof of Theorem \ref{th:main} in Section \ref{sc:main},
and a case-by-case analysis in Sections
\ref{sc:cyclic},
\ref{sc:reflection}, and
\ref{sc:dihedral}.
To construct symmetric and consistent dimer models, we adopt the method
in \cite{Ishii-Ueda_DMSMC}.
The proof of Theorem \ref{th:NCCR}
is given
in Section \ref{sc:NCCR}.
In Section \ref{sc:wallpaper},
we digress from the main subject of this paper
and discuss symmetries of dimer models
under wallpaper groups.

\emph{Acknowledgement}:
We thank the anonymous referee for reading the manuscript carefully
and suggesting many improvements.
A.~I.~ is partially supported
by Grant-in-Aid for Scientific Research
(24540041, 15K04819).
K.~U.~is partially supported
by Grant-in-Aid for Scientific Research
(15KT0105, 16K13743, 16H03930). A.~N.~is supported by the Spanish MINECO (MTM2015-65968-P).

\section{Preliminaries} \label{sc:preliminaries}
\subsection{Dimer models and characteristic polygons}
 \label{sc:dimer}

Let $N$ be a free abelian group of rank 2 and
$M \coloneqq \Hom(N, \bZ)$ be the dual lattice.
We write the real 2-plane
and the real 2-torus
associated with $M$ as
$M_\bR \coloneqq M \otimes \bR$
and
$T \coloneqq M_\bR/M$
respectively.
A \emph{bicolored graph}
$G = (B, W, E)$ on $T$ consists of
\begin{itemize}
 \item
a finite set $B \subset T$ of \emph{black nodes},
 \item
a finite set $W \subset T$ of \emph{white nodes}, and
 \item
a finite set $E$ of \emph{edges},
consisting of embedded closed intervals $e$ on $T$
such that one boundary of $e$ belongs to $B$
and the other boundary belongs to $W$,
\end{itemize}
such that
any edge can intersect another edge
only at its boundary.
The \emph{valence} of a node
is the number of edges adjacent to that node.
A \emph{face} of $G$
is a connected component of $T \setminus \cup_{e \in E} e$.
A bicolored graph $G = (B,W,E)$ on $T$ is a \emph{dimer model}
if
\begin{itemize}
 \item
there is no univalent node, and
 \item
every face of $G$
is simply-connected.
\end{itemize}
Although
a dimer model may have divalent nodes in general,
as explained in \cite[Section 6.1]{Ishii-Ueda_DMSMC},
we can and will assume that
there are no divalent nodes
for the purpose of this paper.
% one can remove any divalent node by merging the two nodes
% connected by a divalent node
% as in \pref{fg:divalent_node}
% (as long as the two nodes are distinct, which is always the case
% for non-degenerate dimer models as explained in \cite[Section 6.1]{Ishii-Ueda_DMSMC}).
% In this paper,
% we assume that a dimer model does not have
% any divalent node
% by performing this operation if necessary.
% As explained in 

% \begin{figure}[h]
% \centering
% \begin{tikzpicture}
% 	\draw (-1,-1) -- (0,0) -- (-1,1) -- (0,0) -- (2,0) -- (3,-1) -- (2,0) -- (3,1);
% 	\draw[fill=white] (0,0) circle (1.2mm);
% 	\draw[fill=black] (1,0) circle (1.2mm);
% 	\draw[fill=white] (2,0) circle (1.2mm);
% 	\draw [dotted,bend right=-45] (-0.5,-0.5) to node[pos=0.5]  {} (-0.5,0.5);
% 	\draw [dotted,bend right=45] (2.5,-0.5) to node[inner sep=1.5pt,pos=0.5]  {} (2.5,0.5);
% 	\draw[->] (4,0) -- (5,0); 

% 	\draw (6,-1) -- (7,0) -- (8,-1) -- (6,1) -- (7,0) -- (8,1);
% 	\draw[fill=white] (7,0) circle (1.2mm);
% 	\draw [dotted,bend right=-45] (6.5,-0.5) to node[inner sep=1.5pt,pos=0.5]  {} (6.5,0.5);
% 	\draw [dotted,bend right=45] (7.5,-0.5) to node[inner sep=1.5pt,pos=0.5]  {} (7.5,0.5);	
% %	\node at (0,0) {\gray 1};
% \end{tikzpicture}
%    \caption{Removing a divalent node.}
%    \label{fg:divalent_node}
% \end{figure}  
% %%%%%%%%%%%%%%%%%%%%%%%%%%%%%

A \emph{perfect matching} is a subset $D \subset E$
of the set of edges such that
for any node $n \in B \sqcup W$,
there is a unique edge $e \in D$
adjacent to $n$.
A dimer model is said to be \emph{non-degenerate}
if any edge is contained in some perfect matching.
For a pair $(D, D_0)$ of perfect matchings,
one can associate an element
$\height(D, D_0) \in H^1(T,\bZ) \cong N$
called the \emph{height change}
(cf. e.g. \cite{Ishii-Ueda_08}).
Fix a perfect matching $D_0$
and call it the \emph{reference matching}.
The lattice polygon $\Delta \subset N_\bR$
obtained as the convex hull
of the set
$
 \lc \height(D, D_0) \relmid D \text{ is a perfect matching} \rc
$
of height changes is called
the \emph{characteristic polygon}.
If we take a different perfect matching $D_1$ as the reference matching,
the resulting characteristic polygon $\Delta'$ is related to $\Delta$
by translation by $h(D_1,D_0)$.

A \emph{zigzag path} is a periodic sequence
$(e_i)_{i \in \bZ}$ of edges,
considered up to translation of $i$,
which makes a maximum turn to the right on a white node
and to the left on a black node.
%We identify zigzag paths
%related by translations of $i$.
A pair of zigzag paths are said to \emph{intersect}
if they share a common edge.
Such an edge will be called an intersection `point'
of the pair of zigzag paths.
The homology class $[z] \in H_1(T,\bZ) \cong M$
of a zigzag path
is called its \emph{slope}.
A dimer model on $T=M_{\bR}/M$ can be pulled back to the universal cover $M_{\bR}$ of $T$ as a doubly periodic bicolored graph and one can consider zigzag paths on $M_{\bR}$.
Zigzag paths on the universal cover will be used when we will define the notion of consistency of a dimer model.

Let $r$ be the number of zigzag paths with non-zero slopes, and
$\{ z_i \}_{i=1}^r$ be the set of such zigzag paths.
A \emph{zigzag polygon} is a convex lattice polygon in $N_\bR$
defined up to translation
by the condition that the multiset of primitive outward normal vectors
to primitive side segments of the polygon
is equal to the multiset
$([z_i])_{i=1}^r$
of slopes of zigzag paths with non-zero slopes.
Here,
a \emph{primitive side segment}
of a lattice polygon
is a line segment on the boundary of the polygon
bounded by a pair of lattice points
containing no lattice point in the interior.
For any dimer model,
the zigzag polygon is contained in the characteristic polygon
\cite[Corollary 1.2]{Beil-Ishii-Ueda_CDM}.
%A pair of lattice polygons is said to \emph{coincide}
%if they are equal up to translation.

%Now we recall the definition of the consistency condition
%for dimer models:

%\begin{definition}[{\cite[Definition 5.2]{Ishii-Ueda_DMSMCv1}}] \label{df:consistency}
A dimer model is {\em consistent}
%\cite{Ishii-Ueda_CCDM,Bocklandt_CCDM}
if
\begin{itemize}
% \item
%there is a perfect matching,
 \item
there is no homologically trivial zigzag path,
 \item
no zigzag path on the universal cover $M_\bR$ of $T$
has a self-intersection, and
 \item
no pair of zigzag paths on the universal cover $M_\bR$
intersect each other
in the same direction more than once.
%\item
%there is a pair of zigzag paths whose slopes are linearly independent.
\end{itemize}
%\end{definition}

\begin{figure}[h]
\centering
\begin{tikzpicture}
	\draw [->,thick,bend right=75] (0,0) to node[inner sep=1.5pt,pos=0.5]  {} (0,3);
	\draw [->,thick,bend right=-75] (1,0) to node[inner sep=1.5pt,pos=0.5]  {} (1,3);

	\draw [<-,thick,bend right=75] (3,0) to node[inner sep=1.5pt,pos=0.5]  {} (3,3);
	\draw [->,thick,bend right=-75] (4,0) to node[inner sep=1.5pt,pos=0.5]  {} (4,3);
\end{tikzpicture}
   \caption{Intersections in the same direction (left) and
    the opposite directions (right).}
   \label{fg:intersection}
\end{figure}

Examples of a pair of curves
intersecting in the same and the opposite direction
are shown in the left and the right of Figure \ref{fg:intersection}
respectively.
%\begin{figure}
%\begin{minipage}{.5 \linewidth}
%\centering
%\input{intersection1.pst}
%\caption{Intersecting in the same direction}
%\label{fg:intersection1}
%\end{minipage}
%\begin{minipage}{.5 \linewidth}
%\centering
%\input{intersection2.pst}
%\caption{Intersecting in the same direction}
%\label{fg:intersection2}
%\end{minipage}
%\end{figure}
See \cite{Ishii-Ueda_CCDM,Bocklandt_CCDM} for more
on consistency conditions for dimer models.
In particular,
it is shown in \cite[Proposition 4.4]{Ishii-Ueda_CCDM}
that a dimer model is consistent
if and only if it is \emph{properly-ordered}
in the sense of Gulotta
\cite{Gulotta}.
%The consistency condition is equivalent to cancellativity:
%
%\begin{theorem}[{\cite[Theorem 1.1]{Ishii-Ueda_CCDM},
%\cite[Theorem 6.2]{Bocklandt_CCDM}}]
%A non-degenerate dimer model is consistent
%if and only if the path algebra of the associated quiver with relations
%is cancellative.
%\end{theorem}
%
Together with \cite[Theorem 3.3]{Gulotta},
this shows that the characteristic polygon and the zigzag polygon coincides
for consistent dimer models.
%
%\begin{theorem}[{\cite[Theorem 3.3]{Gulotta},
% cf. also \cite[Corollary 8.3]{Ishii-Ueda_DMSMCv1}}]
% \label{th:gulotta}
%For a consistent dimer model,
%the characteristic polygon $\Delta$ coincides with the zigzag polygon.
%\end{theorem}
%
A consistent dimer model is non-degenerate
by \cite[Proposition 8.1]{Ishii-Ueda_DMSMC}.

\subsection{Quivers and moduli spaces from dimer models}
 \label{sc:moduli}

%Here should mention 
%\begin{itemize}
%\item Quivers from dimer models and moduli space
%\item Perfect matchings and stability
%\end{itemize}
A \emph{quiver}  $Q = (Q_0, Q_1, s, t)$ consists of
\begin{itemize}
 \item
a set $Q_0$ of vertices,
 \item
a set $Q_1$ of arrows, and
 \item
a pair $s, t \colon Q_1 \to Q_0$ of maps
called the \emph{source} and the \emph{target} respectively.
\end{itemize}
A \emph{path} on a quiver
is either a symbol $e_v$
associated with a vertex $v \in Q_0$
or a sequence $(a_l, \ldots, a_1)$ of arrows
satisfying $s(a_{i+1}) = t(a_i)$ for $i=1,2,\ldots,l-1$.
The \emph{length} of a path is defined to be zero
for $e_v$ and $l$ for $(a_l,\ldots,a_1)$.
The \emph{path algebra} $\bC Q$ of a quiver 
$Q = (Q_0, Q_1, s, t)$
is the algebra
spanned by the set of paths
as a vector space,
and the multiplication is defined
by the concatenation of paths.
Paths of length zero are idempotents of the path algebra,
which sum up to one;
$
 \sum_{v \in Q_0} e_v = 1.
$
A {\em quiver with relations}
is a pair of a quiver
and a two-sided ideal $\cI$
of its path algebra.
For a quiver $\Gamma = (Q, \cI)$
with relations,
its path algebra $\bC \Gamma$ is defined as
the quotient algebra $\bC Q / \cI$.

A dimer model $G = (B, W, E)$ encodes
the information of a quiver with relations
$\Gamma = (Q_0, Q_1, s, t, \cI)$
such that
\begin{itemize}
 \item
$Q_0$ is the set of faces,
 \item
$Q_1$ is the set $E$ of edges,
 \item
the orientations of the arrows are determined
by the colors of the vertices of the graph
in such a way that the white vertex $w \in W$ is on the right
of the arrow, and
\item
the ideal $\cI$ of the path algebra $\bC Q$ is
generated by $p_+(a) - p_-(a)$
for all $a \in Q_1$,
where $p_+(a)$ is the path from $t(a)$ to $s(a)$
going around the white node
adjacent to $a \in E = Q_1$ clockwise, and
$p_-(a)$ is the path from $t(a)$ to $s(a)$
going around the black node
adjacent to $a \in E = Q_1$ counterclockwise.
\end{itemize}

A \emph{representation} of $\Gamma$ is a module
over the path algebra $\bC  \Gamma$.
It is given by a collection
$\Psi=((V_v)_{v \in Q_0}$, $(\psi(a))_{a \in Q_1})$
of vector spaces $V_v$ for $v \in Q_0$
and linear maps $\psi(a) \colon V_{s(a)} \to V_{t(a)}$
for $a \in Q_1$ satisfying relations in $\cI$.
The \emph{dimension vector}
$
 \dim \Psi
$
of a representation
$\Psi=((V_v)_{v \in Q_0}, (\psi(a))_{a \in Q_1})$
is the element
$
 \sum_{v \in Q_0} \lb \dim V_v \rb v
$
of the free $\bZ$-module $\bZ Q_0$
generated by $Q_0$.

Fix a dimension vector $\bfd \in \bZ Q_0$
and a \emph{stability parameter}
$\theta \in \Hom \lb \bZ Q_0, \bZ \rb$
satisfying $\theta(\bfd) = 0$.
A representation $\Psi$ of $\Gamma$
with dimension vector $\bfd$ is \emph{$\theta$-stable}
(resp. \emph{$\theta$-semistable})
if $\theta(\dim S) >0$
(resp. $\theta(\dim S) \ge 0$)
for any non-trivial subrepresentation
$S \subsetneq \Psi$.
The stability parameter $\theta$ is \emph{generic}
if semistability implies stability.

In this paper,
we will always work with the dimension vector
$
 \bsone \coloneqq \sum_{v \in Q_0} v
$
unless otherwise specified.
For a vertex $v_0 \in Q_0$,
a stability parameter $\theta$ is \emph{$v_0$-generated}
if $\theta(v)>0$ for any $v \ne v_0$.
Any $v_0$-generated parameter $\theta$ is always generic, and
a representation $\Psi$
with dimension vector $\bsone$
is $\theta$-stable
if and only if $\Psi$ is generated by a non-zero element in $V_{v_0}$
as a module over $\bC\Gamma$.

Let $\Delta$ be the characteristic polygon
of a dimer model $G$ and
$X_\Delta \coloneqq \Spec R$ be the Gorenstein affine toric 3-fold,
whose coordinate ring $R$
is the monoid ring $\bC[C(\Delta)^\dual\cap (M\oplus \bZ)]$
of the dual cone
of the cone $C(\Delta)$ over
$\Delta \times \{ 1 \} \subset N_\bR \times \bR$.
Here we fix an embedding $\Delta \subset N_\bR$ so that
it contains the origin of $N$.
The dense torus of $X_\Delta$ will be denoted by $\bT$.
If $G$ is consistent and $\theta$ is generic,
then the moduli space $\cM_\theta$ of $\theta$-stable representations
with dimension vector $\bsone$ is
a $\bT$-equivariant crepant resolution
of
%the affine toric 3-fold 
$X_\Delta$
by \cite[Theorem 6.4]{Ishii-Ueda_08}.
Toric divisors in $\cM_\theta$ correspond to perfect matchings on $G$ \cite[Section 6]{Ishii-Ueda_08},
and we write the perfect matching corresponding
to the toric divisor associated with the origin of $\Delta$
as $D_0$.
The corresponding one-parameter subgroup
of $\bT$
will be denoted by
\begin{equation}\label{eq:one-para}
\lambda_0 \colon \bCx \to \bT.
\end{equation}

The moduli space $\cM_\theta$ is equipped with  the tautological
bundle $\bigoplus_{v \in Q_0} \cL_v$ which,
by \cite[Theorem 1.4]{Ishii-Ueda_DMSMC}, is a tilting bundle on $\cM_\theta$ with
\[
\End\left(\bigoplus_{v \in Q_0} \cL_v\right) \cong \bC\Gamma.
\]
Notice that the tautological bundle is determined only
up to tensor product by a line bundle on $\cM_\theta$.

%%%%%%%%%%%%%%%%%%
\section{Group actions on dimer models}
 \label{sc:action}

A finite subgroup $H$ of $\GL(N)$ acts contragradiently on $M$,
and hence on $T \coloneqq M_\bR/M$.

\begin{definition} \label{df:action}
A dimer model $G$ on $T$ is \emph{symmetric}
with respect to the action of $H$
if for all $h \in H$ we have that:
\begin{itemize}
 \item
$h$ preserves the set $E$,
 \item
$h$ preserves the sets $B$ and $W$ individually
if $\det h = 1$, and
 \item
$h$ exchanges $B$ and $W$ if $\det h = -1$.
\end{itemize}
\end{definition}

\begin{remark}
Recall that the sets $B$, $W$, and $E$ are subsets of $T$
in our definition of a dimer model.
The action of $h$ on $G$ is required to preserve $E$ as a subset of $T$,
and similarly for $B$ and $W$.
In particular,
the symmetry of a dimer model in our sense
depends on how the graph is embedded in $T$.
\end{remark}

Examples of symmetric dimer models can be found
in Figures \ref{fg:Rotation4} and \ref{fg:Hexagon} below.
Here the origin of $T$ is the center of one octagonal face in Figure \ref{fg:Rotation4} 
and the center of the dodecagonal face in Figure \ref{fg:Hexagon}.

The conditions in Definition \ref{df:action} ensure that
if a dimer model $G$ is symmetric with respect to the action of $H$,
then $H$ acts on the quiver $\Gamma = (Q, \cI)$ with relations associated with $G$.
On the other hand, for perfect matchings $D_1$ and $D_2$, recall that the height change $\height(D_1, D_2)$ is defined as an element of $N$ independently of the choice of a basis of $N$ which is acted on by $H$. 
Then one can see $\height(h(D_1), h(D_2))=h\height(D_1, D_2)$ holds for any $h \in H$.
Thus $h$ sends the set $\{\height(D, D_0) \mid \text{$D$ is a perfect matching}\}$
to the set $\{\height(D, h(D_0)) \mid \text{$D$ is a perfect matching}\}$
and hence with respect to the linear action of $h \in H \subset \GL(N)$ on $N_\bR$, $h(\Delta)$
equals the translation of $\Delta$ by $\height(h(D_0), D_0)$.
Thus, with respect to this linear action, $\Delta$ is fixed by $H$ if $\height(h(D_0), D_0)=0$.

We will make the following assumptions
throughout this paper:

\begin{assumption} \label{as:polygon}
The characteristic polygon $\Delta$ is fixed by the action of $H$
for a suitable choice of a reference perfect matching.
\end{assumption}

\begin{assumption} \label{as:vertex}
There is a vertex $v_0$ of $Q$ fixed by the action of $H$.
\end{assumption}

In particular,
the symmetric dimer model
obtained in our proof of Theorem \ref{th:main}
satisfies Assumptions \ref{as:polygon} and \ref{as:vertex}.

Assumption \ref{as:polygon} means that
the height change
$\height(h(D_0), D_0) \in N$
of the image of the reference matching $D_0$
by any $h \in H$
is zero as noted above.
In this case, we obtain a torus-equivariant action
\[
\mu: H \times X_\Delta \to X_\Delta
\]
on the affine toric variety $ X_\Delta$
associated with the cone over $\Delta$.

\begin{remark}
By combining the translation by $\height(D_0, h(D_0))$ to the action of $h$ on $N$ for each $h \in H$,
we obtain an affine linear action of $H$ on $N$ which preserves $\Delta \subset N_{\bR}$.
This affine linear action on $N$ can be extended to a linear action on $N \times \bZ$
whose restriction to $N \times \{1\}$ coincides with the affine linear action.
Thus we can define the action of $H$ on $X_\Delta$ even if there is no perfect matching $D_0$ with $\height(h(D_0), D_0)=0$.
However, we do not consider such a situation in this paper.
\end{remark}
Since the pull-back by
$
 \mu(h,-) \colon X_\Delta \to X_\Delta
$
acts by multiplication by $\det(h)$
on the canonical module $\omega_{X_\Delta}$
of the Gorenstein affine toric variety $X_\Delta$
associated with $\Delta$,
the line bundle $\omega_{X_\Delta}$ is not $H$-equivariantly trivial with respect to the action $\mu$ of $H$
if $H$ is not contained in $\SL(N)$.
In that case, the fixed point locus of a reflection is a divisor
which is the closure of a codimension one subtorus of $\bT \subset X_\Delta$.
In order to make the action of $H$ small (i.e., free in codimension one)
and to obtain a Gorenstein singularity as the quotient,
we twist the action of $H$ on $X_\Delta$
by the one-parameter subgroup $\lambda_0$ of $\bT$ (see \eqref{eq:one-para}) as
\begin{align}\label{eq:twist_on_X}
 \nu(h,x) = \lambda_0(\det(h)) \cdot \mu(h,x),
\end{align}
so that the induced action on the canonical module is trivial.
Note that the action $\nu$ of $H$ depends
on the choice of the origin of $\Delta$,
although $X_\Delta$ as an abstract variety does not.

\begin{remark}\label{rem:choice}
The twisted action $\nu$ in \eqref{eq:twist_on_X} depends on the choice of the origin in $\Delta$
when $H\cong\bZ/2\bZ$ is a reflection group of order $2$.
If $\Delta$ is a lattice triangle, we recover the dihedral groups in $\SL(3,\bC)$ acting on $\bC^3$. A \emph{dihedral group} in $\SL(3,\bC)$ is obtained by the natural embedding of a dihedral group $G \subset \GL(2,\bC)$ into $\SL(3,\bC)$, which belongs  to the type B family in the Yau-Yu classification \cite{YY93}. 
Recall that for $1<q<m$ with $(m,q)=1$,
a \emph{dihedral group in $\GL(2,\bC)$} is defined as
\[
G=\mathbb{D}_{m,q}:= \left\{ \begin{array}{ll} 
\Span{\psi_{2q}, \tau, \phi_{2k}}, & \text{if }k:=m-q\equiv1\text{ mod 2} \\ 
\Span{\psi_{2q}, \tau\circ\phi_{4k}}, & \text{if }k:=m-q\equiv0\text{ mod 2} 
\end{array} \right.
\]
with matrices $\psi_{r}=\begin{pmatrix}\varepsilon_{r}&0\\0&\varepsilon_{r}^{-1}\end{pmatrix}, \tau=\begin{pmatrix}0&\varepsilon_4\\\varepsilon_4&0\end{pmatrix}, \phi_{r}=\begin{pmatrix}\varepsilon_{r}&0\\0&\varepsilon_{r}\end{pmatrix}$, where $\varepsilon_{r}$ is a primitive $r$-th root of unity. Every such group can be described as a finite group with an index 2 abelian subgroup $A$ (see \cite[Remark 3.3]{NdC12}). For example, if $A$ is cyclic, then $G=\Span{\alpha,\beta}$ where $\beta^2\in \Span{\alpha}=A$. In what follows we assume for simplicity that $A$ is cyclic, although the arguments also work for a general abelian group.

A triangle $\Delta$ which admits a reflection can be embedded as the junior simplex (i.e.\ the triangle with vertices the standard basis $e_1$, $e_2$, $e_3$) of the cyclic group $A=\frac{1}{n}(1,a,-(a+1))$ with $a^2\equiv1$ mod $n$, where $N$ is identified with
$$
\left(\bZ^3 + \bZ \frac{1}{n}(1,a,-(a+1))\right) \cap \{(x, y, z) \mid x+y+z=1\},
$$
up to the choice of the origin, and the reflection happens along the plane $x=y$.
Then the action of $H$ on $\Delta$ can be lifted to the action on $X_\Delta$ by the matrix $\mu=\left(\begin{smallmatrix}0&1&0\\1&0&0\\0&0&1\end{smallmatrix}\right)$, which is not trivial on the canonical module $\omega_{X_\Delta}$.
The points in $\Delta\cap N$ fixed by $H$ are of the form $P_j=(\frac{jq}{n},\frac{jq}{n},1-\frac{2jq}{n})$ where $q=\frac{n}{(a-1,n)}$ for $0\leq j\leq\floor{\frac{n}{2q}}$, with corresponding one-parameter subgroups $\lambda_{P_j}:\mathbb{C}^*\to\mathbb{T}$ of the form $\lambda_{P_j}(t)=(t^{jq/n},t^{jq/n},t^{1-2jq/n})$. 
Here note that one-parameter subgroups of $\bT$ are identified
with elements of $N\subset \bQ^3$ since $\bT = (\bCx)^3/A$ 
and the one-parameter subgroup corresponding to $(a_1, a_2, a_3) \in N$ is denoted by
$t \mapsto (t^{a_1}, t^{a_2}, t^{a_3})$ even though $a_i$ are rational numbers.
In particular $\lambda_{P_j}(-1)=(\varepsilon^{jq/2},\varepsilon^{jq/2},-\varepsilon^{-jq})$ in $\bT$, and taking $P_j$ as the origin of $\Delta$ we have that $X_\Delta/H\cong\bC^3/G_j$ where $G_j=\Span{\frac{1}{n}(1,a,-(a+1)),\lambda_{P_j}(-1)\cdot\mu}$ is a dihedral group. It can be shown that $G_0\cong G_{2i}$ and $G_1\cong G_{2i+1}$, which implies that there are at most two non-isomorphic dihedral actions on $\bC^3$ associated to $\Delta$ given by
\begin{align*}
G_0 &= \Span{\frac{1}{n}(1,a,-(a+1)), \left(\begin{smallmatrix}0&1&0\\1&0&0\\0&0&-1\end{smallmatrix}\right)% ~|~\text{$a^2\equiv1$ mod $n$}
},\\ 
G_1 &= \Span{\frac{1}{n}(1,a,-(a+1)), \left(\begin{smallmatrix}0&\varepsilon^{q/2}&0\\\varepsilon^{q/2}&0&0\\0&0&-\varepsilon^{-q}\end{smallmatrix}\right) %~|~ \text{$\varepsilon=e^{2\pi i/n}$, $a^2\equiv1$ mod $n$}
},
\end{align*}
where $a^2\equiv1$ mod $n$ and $\varepsilon=e^{2\pi\sqrt{-1}/n}$.

In general, $G_0$ and $G_1$ may be isomorphic and every dihedral subgroup $G\subset\SL(3,\bC)$ can be written in this form. We note that in the case when $a=n-1$ then $G_0\cong\dD_n\subset\SO(3)$, and if $n\geq4$ is even then $G_0\cong\BD_{n}$, where $\dD_n=\Span{\alpha,\beta~|~\alpha^{n}=\beta^2=1,\alpha\beta=\beta\alpha^{-1}}$ and $\BD_{n}=\Span{\alpha,\beta~|~\alpha^{n}=1, \beta^2=\alpha^{n/2},\alpha\beta=\beta\alpha^{-1}}$ are the dihedral and the binary dihedral groups respectively, both in the ``classical" sense.
\end{remark}

\begin{example} The triangle $\Delta$ formed as the junior simplex for the subgroup $\frac{1}{12}(1,7,4)$ admits the above two non-isomorphic dihedral actions, where $G_1\cong D_{5,2}$ in the Yau-Yu notation (see \cite[B1, p.12]{YY93}). The group $G_0$ is not included in the Yau-Yu classification since the (isomorphic) group $\widetilde{G}=\Span{\frac{1}{12}(1,7),\left(\begin{smallmatrix}0&1\\1&0\end{smallmatrix}\right)}\subset\GL(2,\bC)$ is not small.  
\end{example}

Correspondingly, we have to twist
the action of $H$ on the path algebra $\bC \Gamma$.
\begin{lemma}
Under Assumptions \ref{as:polygon} and \ref{as:vertex},
there exists a perfect matching $D_0$ which is fixed by the action of $H$.
\end{lemma}
\begin{proof}
Notice that the action of $H$ on $Q_0$ induces an action of $H$ on the set of stability parameters.
Then there exists a $v_0$-generated stability parameter $\theta$ fixed by this action.
Let $D_0$ be the $\theta$-stable perfect matching corresponding to the origin.
Then it is easy to see that $D_0$ is fixed by the $H$-action.
\end{proof}

Using the invariant perfect matching $D_0$, we twist the natural action of $H$ on $\bC Q$  as
\begin{align}\label{eq:twist_on_Gamma}
 a \mapsto
\begin{cases}
 \det(h) h(a) & a \in D_0, \\
 h(a) & \text{otherwise}.
\end{cases}
\end{align}
Notice that this twist preserves the relation and thus gives an action of $H$ on $\bC\Gamma=\bC Q/\scI$.

\begin{definition}
Let $G$ be a finite group
acting on a ring $A$ from the left
by a homomorphism $\varphi \colon G \to \Aut A$.
The \emph{crossed product algebra}
$
A \rtimes_\varphi G
$
is the vector space
$
A \times G
\cong
A \otimes_\bC \bC[G]
$
equipped with the product
$
(a, g) \cdot (a', g')
\coloneqq
(a \varphi(g)(a'), g g').
$
Similarly,
for a finite group $G$
acting on a ring $A$ from the right
by a homomorphism $\psi \colon G^{\mathrm{op}} \to \Aut A$,
the \emph{crossed product algebra}
$
G \ltimes_\psi A
$
is the vector space
$
G \times A
$
equipped with the product
$
(g, a) \cdot (g', a')
\coloneqq
(g g', \psi(g')(a) a').
$
We drop $\varphi$ and $\psi$ from the notation
when they are clear from the context.
\end{definition}

\begin{remark}
The pre-composition of the anti-automorphism
$G \to G^\mathrm{op}$
sending $g$ to $g^{-1}$
gives a bijection
$\Hom(G, \Aut A) \to \Hom(G^\mathrm{op}, \Aut A)$.
If $\varphi$ and $\psi$ are related by this bijection,
then one has an isomorphism $A \rtimes_\varphi G \to G \ltimes_\psi A$
sending $(a,g)$ to $(g, \psi(g)(a))$.
\end{remark}

In order to give a quiver with relations
which is Morita equivalent to the crossed product algebra
$H \ltimes \bC \Gamma$,
choose a complete representative
$Q_0' \subset Q_0$
of $Q_0/H$.
The $H$-orbit and the stabilizer subgroup
of $v \in Q_0'$
will be denoted by
$
 O_v \coloneqq H \cdot v \subset Q_0
$
and
$
 H_v \subset H
$
respectively.
Since $H$ is a principal $H_v$-bundle over $O_v$,
the category of $H$-equivariant vector bundles on $O_v$ is equivalent
to the category of $H_v$-equivariant vector bundles on $v$.
In other words,
the crossed product algebra
$
 H \ltimes \bC[O_v]
$
of $H$
with the algebra
\begin{align}
 \bC[O_v]
  \coloneqq \bigoplus_{w \in O_v} \bC \, e_w
  \subset \bC \Gamma
\end{align}
of functions on $O_v$
is Morita equivalent to the group algebra
$\bC[H_v]$ of $H_v$.
A classical result in representation theory of finite groups gives
a ring isomorphism
\begin{align}
 \bC[H_v] \cong \bigoplus_{\rho \in \Irrep(H_v)} \End_\bC(\rho).
\end{align}
%It follows that $H \ltimes \bC[O_v]$ is isomorphic
%to the direct sum of matrix algebras;
%\begin{align}
% H \ltimes \bC[O_v] \cong
%  \bigoplus_{\rho \in \Irrep(H_v)} \End_\bC \lb \bC^{n_\rho} \rb.
%\end{align}
Choose a primitive idempotent $e_\rho$
in the matrix algebra $\End_\bC \lb \rho \rb$
for each $\rho \in \Irrep(H_v)$
and set
\begin{align}
 e = \sum_{v \in Q_0'} \sum_{\rho \in \Irrep(H_v)} e_\rho.
\end{align}
Then $e \lb H\ltimes\bC \Gamma \rb e$ is Morita equivalent to $H\ltimes\bC \Gamma$,
and $\{ e_\rho \}_\rho$ gives a set of mutually orthogonal idempotents
in $e \lb H\ltimes\bC \Gamma \rb e$
which sum up to the identity.
This allows one to describe $e \lb H\ltimes\bC \Gamma \rb e$
in terms of a quiver with relations;
the set $V$ of vertices is $\bigcup_{v \in Q_0'} \Irrep(H_v)$,
and for each (not necessarily distinct) pair $(\rho, \rho')$ of vertices,
we choose a finite subset of $e_{\rho'} (H \ltimes \bC \Gamma) e_\rho$
as the set of arrows from $\rho$ to $\rho'$,
in such a way that the union for all pairs
generate $e \lb H\ltimes\bC \Gamma \rb e$
as an algebra.

To illustrate the constructions so far,
we discuss two-dimensional examples,
which are simpler than, but shares the essential features
of, three-dimensional cases.

\begin{example}
%
%While a dimer model is a bicolored graph on a 2-torus
%giving a polygon division of the torus,
A two-dimensional analog of a dimer model
is a collection of uncolored nodes on a circle,
which divides the circle into intervals.
The division of the circle into $n$ intervals corresponds to
the McKay quiver $\Gamma = (Q_0,Q_1,s,t,\cI)$
for the subgroup $A$ of $\SL(2,\bC)$
generated by $\gamma \coloneqq \diag \lb \zeta_n, \zeta_n^{-1} \rb$,
where $\zeta_n \coloneqq \exp \lb 2 \pi \sqrt{-1}/n \rb$.
The set $Q_0$ of vertices consists
of irreducible representations
$\rho_i \colon \gamma \mapsto \zeta_n^i$
of $A$
for $i = 0, 1, \ldots, n-1$.
The set of arrows consists of $x_i$ and $y_i$
for $i = 0, 1, \ldots, n-1$
with sources $s(x_i) = \rho_i$, $s(y_i) = \rho_i$
and targets $t(x_i) = \rho_{i+1}$, $s(y_i) = \rho_{i-1}$,
and the ideal of relations are generated by
$
 x_{i-1} y_i - y_{i+1} x_i
$
for $i = 0, 1, \ldots, n-1$.
The path algebra
$
\bC \Gamma
$
can be identified with the crossed product algebra
$
A \ltimes \bC[x,y]
$
in such a way that
\begin{itemize}
\item
the idempotent of the path algebra $\bC \Gamma$
corresponding to the vertex $\rho_i \in Q_0$
is identified with
the idempotent
$
e_i
= \frac{1}{n} \sum_{j=0}^{n-1} \zeta_n^{-ij} \gamma^j
$
of the group ring
$
\bC[A]
\subset A \ltimes \bC [x,y]
$
corresponding to the projection to $\rho_i$,
and
\item
$x_i = e_{i+1} x e_i$ and $y_i = e_{i-1} y e_{i}$,
so that
$
x = \sum_{i=0}^{n-1} x_i
$
and
$
y = \sum_{i=0}^{n-1} y_i.
$
\end{itemize}
The analog of the characteristic polygon in this case
is the interval $\Delta$ of length $n$ in $N_\bR \coloneqq N \otimes \bR$
where $N$ is a free abelian group of rank 1,
and the associated toric variety $X_\Delta$
gives the $A_n$-singularity $\bC^2/A$.
The cyclic group $H$ of order two
is the only non-trivial finite subgroup of $\GL(N)$.
The induced action $\mu$ of $H$ on $X_\Delta$
does not preserve the canonical module,
and one can twist the action
to obtain a Gorenstein quotient singularity
only if $n$ is even.
This condition on the parity of $n$
is ensured by Assumption \ref{as:polygon}.
The quotient of $X_\Delta$ by the twisted action of $H$
is the quotient of $\bC^2$ by the binary dihedral group $\BD_{n}$ of order $2n$.
%which is the semidirect product $H \ltimes A$.
For even $n$,
there are two ways to make $H$ act on the circle $M_\bR / M$.
One fixes a pair of intervals and acts non-trivially on the remaining $n-2$ intervals,
and the other acts non-trivially on all the intervals.
Only the former satisfies Assumption \ref{as:vertex}.
%and gives the McKay quiver for the dihedral group.
%
 Let us consider the case $n=4$, i.e.\ the group $\BD_{4}$ of order 8.
The action of the generator $\sigma$ of $H \cong \bZ/2\bZ$
%$
% H \coloneqq \la 1, \sigma \ra \cong \bZ/2\bZ
%$
on the McKay quiver $\Gamma$
fixes the vertices $\rho_0$ and $\rho_2$, and
interchanges the vertices $\rho_1$ and $\rho_3$.
The action on the arrows depends on a choice
of a perfect matching.
Choosing a perfect matching corresponds to
choosing one arrow from each of the pairs
$\{x_0, y_1\}, \{x_1, y_2\}, \{x_2, y_3\}, \{x_3, y_0\}$.
The choice $y_1, y_2, x_2, x_3$ corresponds to the $0$-generated
$H$-invariant perfect matching, with respect to which
the action of $\sigma$ on the arrows is given by
\begin{equation}\label{eq:sigma}
\begin{aligned}
 x_0 &\leftrightarrow y_0 \\
 x_1 &\leftrightarrow y_3 \\
 x_2 &\leftrightarrow -y_2 \\
 x_3 &\leftrightarrow -y_1.
\end{aligned}
\end{equation}
The path algebra $\bC \Gamma$ with relations is isomorphic
to the crossed product algebra
$
 A \ltimes \bC[x,y],
$
and
%the crossed product algebra
$
 B \coloneqq H \ltimes \bC \Gamma
$
is isomorphic to
$
 \BD_4 \ltimes \bC[x,y],
$
where
$$
\BD_4=\Span{\frac{1}{4}(1,3), \left(\begin{smallmatrix}0&\sqrt{-1}\\\sqrt{-1}&0\end{smallmatrix}\right)}
$$
is the binary dihedral group of type $D_4$
and the matrix $\left(\begin{smallmatrix}0&\sqrt{-1}\\\sqrt{-1}&0\end{smallmatrix}\right)\in \BD_4$
corresponds to $\delta:=\sigma((e_0-e_2) +\sqrt{-1}(e_1+e_3)) \in B$.
In fact,  $\delta^2$ is identified with $\gamma^2\in A$ and Equation \eqref{eq:sigma} implies $\delta x_i \delta^{-1}=\sqrt{-1} y_{-i}$ and $\delta y_i \delta^{-1}=\sqrt{-1} x_{-i}$ in $B$.
The algebra $B$ has primitive idempotents
\begin{align}
 e_{00} &= \frac{1}{2}(1+\sigma) e_0, \\
 e_{01} &= \frac{1}{2}(1-\sigma) e_0, \\
 e_{130} &= \frac{1}{2}(1+\sigma) (e_1+e_3), \\
 e_{131} &= \frac{1}{2}(1-\sigma) (e_1+e_3), \\
 e_{20} &= \frac{1}{2}(1+\sigma) e_2, \\
 e_{21} &= \frac{1}{2}(1-\sigma) e_2
\end{align}
which are mutually orthogonal 
and sum up to the identity.
The projective modules 
$
 P_{130} = e_{130} B
$
and
$
 P_{131} = e_{131} B
$
are isomorphic as $B$-modules
by the map
\begin{align}
 m \mapsto (e_1-e_3) \cdot m.
\end{align}
Indeed,
this map interchanges $P_{130}$ and $P_{131}$
since
\begin{align}
 (e_1-e_3) (1 \pm \sigma) = (1 \mp \sigma) (e_1-e_3),
\end{align}
and it is an isomorphism since
\begin{align}
  (e_1-e_3)^2=(e_1+e_3).
\end{align}
Therefore,
one can choose
\begin{align}
 e = e_{00} + e_{01} + e_{130} + e_{20} + e_{21}.
\end{align}
%One can easily see that
%the primitive part of $e_{130} B e_{00}$ is represented by
One can take
\begin{align}
 (1+\sigma)(e_1+e_3)x(1+\sigma)e_0
  &=(1+\sigma)(x_0 + \sigma y_0)\\
  &= (1+\sigma)(\sigma x_0 + y_0) \\
  &= (1+\sigma)(e_1+e_3)y(1+\sigma)e_0,
\end{align}
as the element corresponding to a unique arrow
from $e_{00}$ to $e_{130}$.
Arrows between other vertices
can be computed similarly, which generates $B$ as an algebra.
Moreover, one can deduce the relation for the McKay quiver
for the binary dihedral group $\BD_4$
from the relations for the McKay quiver for the cyclic group $A$.

\end{example}

\section{Finite subgroups of $\GL(2,\bZ)$}
 \label{sc:subgroups}

Finite subgroups of $\GL(2,\bZ)$ are classified as follows:

\begin{proposition} \label{pr:H}
A finite non-trivial subgroup of $\GL(2,\bZ)$ is conjugate
to one of the following:
\begin{enumerate}
\item
Cyclic group of rotations:
\begin{itemize}
\item $C_2 = \left\langle \begin{pmatrix} -1 & 0 \\ 0 & -1 \end{pmatrix}\right\rangle$ of order $2$.
\item $C_3=\left\langle \begin{pmatrix} 0 & -1 \\ 1 & -1 \end{pmatrix} \right\rangle$ of order $3$.
\item $C_4=\left\langle \begin{pmatrix} 0 & -1 \\ 1 & 0 \end{pmatrix}\right\rangle$ of order $4$.
\item $C_6=\left\langle \begin{pmatrix} 1 & -1 \\ 1 & 0 \end{pmatrix}\right\rangle$ of order $6$.
\end{itemize}
\item
Reflection groups of order $2$:
\begin{itemize}
\item $R_1 = \left\langle \begin{pmatrix} 1 & 0 \\ 0 & -1 \end{pmatrix}\right\rangle$.
\item $R_2 = \left\langle \begin{pmatrix} 0 & 1 \\ 1 & 0 \end{pmatrix}\right\rangle$.
\end{itemize}
\item
Dihedral groups:
\begin{itemize}
\item $D_4^1 = \left\langle \begin{pmatrix} -1 & 0 \\ 0 & -1 \end{pmatrix}, \begin{pmatrix} 1 & 0 \\ 0 & -1 \end{pmatrix}\right\rangle$ of order $4$.
\item $D_4^2 = \left\langle \begin{pmatrix} -1 & 0 \\ 0 & -1 \end{pmatrix}, \begin{pmatrix} 0 & 1 \\ 1 & 0 \end{pmatrix}\right\rangle$ of order $4$.
\item $D_6^1=\left\langle \begin{pmatrix} 0 & -1 \\ 1 & -1 \end{pmatrix}, \begin{pmatrix} 1 & 0 \\ 1 & -1 \end{pmatrix}\right\rangle$ of order $6$.
\item $D_6^2=\left\langle \begin{pmatrix} 0 & -1 \\ 1 & -1 \end{pmatrix}, \begin{pmatrix} 0 & 1 \\ 1 & 0 \end{pmatrix}\right\rangle$ of order $6$.
\item $D_8=\left\langle \begin{pmatrix} 0 & -1 \\ 1 & 0 \end{pmatrix}, \begin{pmatrix} 1 & 0 \\ 0 & -1 \end{pmatrix}\right\rangle$ of order $8$.
\item $D_{12}=\left\langle \begin{pmatrix} 1 & -1 \\ 1 & 0 \end{pmatrix}, \begin{pmatrix} 0 & 1 \\ 1 & 0 \end{pmatrix}\right\rangle$ of order $12$.
\end{itemize}
\end{enumerate}
\end{proposition}

\begin{proof}
Let $A \in \GL(2,\bZ)$ be an element of finite order $m$.
Since the characteristic polynomial of $A$ is of degree $2$ and is divisible by the $m$-th cyclotomic polynomial,
we see that $m$ is either $1$, $2$, $3$, $4$ or $6$.

A finite subgroup $H$ of $\GL(2,\bZ)$
is either cyclic or dihedral,
since $H$ is conjugate to a subgroup of $O(2)$.

If $H$ is cyclic of order greater than $2$,
then $H$ is a rotation group.
Consider an $H$-invariant metric on $\bR^2$ and
take a vector $v \in \bZ^2 \setminus 0$ with the smallest length.
Then there are no other lattice points in the triangle formed by $0$, $v$ and $Av$,
so that $v$ and $Av$ form a $\bZ$-basis of $\bZ^2$,
and $H$ is conjugate to $C_3$, $C_4$ or $C_6$ above.
If $H$ is a rotation group of order $2$, then $H=\langle -1 \rangle$.

If $H=\langle A \rangle$ is a reflection group of order $2$,
then take two primitive vectors
$v, w \in \bZ^2$ with $Av=v$ and $Aw=-w$.
If $v, w$ form a $\bZ$-basis of $\bZ^2$, then $H$ is conjugate to $R_1$.
Otherwise, there is an integral vector $u=\alpha v + \beta w \in \bZ^2$ with $\alpha, \beta \in (0,1)$.
Then the equations $u+Au=2\alpha v \in \bZ^2$ and $u-Au=2\beta w \in \bZ^2$ imply $\alpha=\beta=1/2$.
Thus $H$ is conjugate to $R_2$.

If $H$ is dihedral of order $4$,
then $H$ is generated by a reflection and $-1$,
so that it is conjugate to either $D_4^1$ or $D_4^2$.
In the remaining cases, consider an $H$-invariant metric and take a vector $v \in \bZ^2 \setminus 0$ of the smallest length.
If $H$ is dihedral of order $6$ and $A \in H$ is a rotation of order $3$,
then $v$ and $Av$ form a $\bZ$-basis of $\bZ^2$ and $\pm v$, $\pm Av$, $\pm A^2v$ are the
non-zero integral vectors of the smallest length.
Therefore, $H$ preserves the hexagon whose vertices are these six vectors.
It follows that $H$ is conjugate to $D_6^2$
if $H$ preserves the triangle formed by $v$, $Av$ and $A^2v$,
and conjugate to $D_6^1$ otherwise.
If $H$ is dihedral of order $8$ and $A \in H$ is a rotation of order $4$,
then $\pm v, \pm Av$ are the non-zero vectors of the smallest length
and $H$ preserves the square formed by these $4$ vectors.
Therefore $H$ is conjugate to $D_8$.
Similarly, a dihedral group of order $12$ is conjugate to $D_{12}$.
\end{proof}

\section{Construction of symmetric dimer models}
 \label{sc:main}

Let $H$ be a finite subgroup of $\GL(N)$ and
$\Delta$ be an $H$-invariant lattice polygon in $M_\bR$.
A \emph{corner} of $\Delta$ is a point on the boundary of $\Delta$
such that $\Delta$ is not defined by one linear inequality
in any neighborhood of that point.
Our strategy for constructing a symmetric dimer model is the following:

\begin{enumerate}[(1)]
 \item
Embed $\Delta$ into an $H$-invariant polygon $\Deltatilde$,
which is the characteristic polygon of a consistent symmetric dimer model $\Gtilde$. 
To find such a dimer model $\Gtilde$,
we enlarge a small example by a linear transform by Lemma \ref{lm:UY}, and
cut off its corners by using Proposition \ref{pr:Gulotta} if necessary.
 \item
If there exists a corner $\frakc$ of $\widetilde{\Delta}$ not in $\Delta$,
then remove the orbit $H\cdot \frakc$ and take the convex hull of the rest.
When we consider only one corner,
this corresponds to removing edges in the dimer model $\Gtilde$
using the special McKay correspondence
as in \cite{Ishii-Ueda_DMSMC}. Proposition 
\ref{pr:chop0} allows us to do the operations symmetrically,
%with respect to the action of $H$,
under some conditions on $\widetilde{\Delta}$.
 \item
Repeat the second step until we obtain $\Delta$.
\end{enumerate}

The dimer model $\Gtilde$ in the first step must be constructed
so that the lattice polygon $\Deltatilde$ satisfies the conditions
in Proposition \ref{pr:chop0} in each step of corner removal.
%We construct such $\Gtilde$ in each case
%depending on $H$ classified in \pref{pr:H}.
%At the moment we can prove a statement as \pref{pr:reflection} when $G$ is generated by a rotation of order 2, 3 and 4. Rotation of order 6 is still remaining.
%\subsection{Constructing a symmetric dimer model $\Gtilde$}

To find a suitable polygon $\Deltatilde$ and a dimer model $\Gtilde$,
first note the following obvious fact:

\begin{lemma} \label{lm:UY}
Let $G$ be a consistent dimer model on $T=M_\bR/M$
whose characteristic polygon is $\Delta \subset N_\bR$, and
$\Mtilde$ be a sublattice of $M$ of finite index.
Then the $M/\Mtilde$-cover $\Gtilde$ of $G$
on $\Ttilde \coloneqq \Mtilde_\bR/\Mtilde \cong M_\bR/\Mtilde$
is a consistent dimer model,
whose characteristic polygon $\Deltatilde$ is $\Delta$
considered as a lattice polygon in $\Ntilde_\bR \cong N_\bR$.
\end{lemma}

In other words,
a similarity transformation of the characteristic polygon
is obtained by changing the fundamental domain of the dimer model,
and this operation preserves the consistency.
If $G$ is symmetric with respect to the action of $H$
and $\Mtilde$ is invariant under $H$,
then $\Gtilde$ is also symmetric with respect to the action of $H$.

%\begin{corollary}
%Let $G$ be a consistent dimer model on which a finite subgroup $H \subset \GL(2, \bZ)$ acts
%and let $\Delta$ be its characteristic polygon.
%If $H$ commutes with a matrix $P \in M_2(\bZ)$ with $\det P >0$
%then there is a consistent dimer model $G^P$ with $H$-action whose characteristic polygon is $P^T(\Delta)$.
%\end{corollary}

We also use Proposition \ref{pr:Gulotta} below
to construct a symmetric dimer model $\Gtilde$
in some cases.

\begin{proposition}[\cite{Gulotta}] \label{pr:Gulotta}
Let $G$ be a consistent dimer model with characteristic polygon $\Delta$
and $\frakc$ be a corner of $\Delta$.
Let further $\frakd$ and $\frakd'$ be the pair of corners of $\Delta$ adjacent to $\frakc$, and
$z_1, \dots, z_l$ and $z_1', \dots, z_m'$ be zigzag paths of $G$ whose slopes are outer
normal to the sides $\frakc\frakd$ and $\frakc\frakd'$ respectively.
Take the $l$-th lattice point $R$ on $\frakc\frakd$
and the $m$-th lattice point $R'$ on $\frakc\frakd'$ counted from $\frakc$.
Let $G'$ be the bicolored graph obtained by removing all the intersections of $z_i$ and $z_j'$
for all $(i, j)$.
If $\Delta$ does not coincide with the triangle formed by the lattice points $\frakc$, $R$ and $R'$,
then $G'$ is a consistent dimer model whose characteristic polygon is the polygon $\Delta'$
obtained from $\Delta$ by removing the triangle $\frakc RR'$.
\end{proposition}

\begin{proof}
Since $\Delta$ does not coincide with the triangle $\frakc RR'$,
$G$ has a pair of zigzag paths other than $z_i$ or $z_j'$
whose slopes are linearly independent.
These zigzag paths remain in the resulting bicolored graph $G'$,
and hence $G'$ is a dimer model.
The operation creates several new zigzag paths,
consisting of edges in $\bigcup z_i \cup \bigcup z'_j$.
The slopes of new zigzag paths
are the outward normal vector to the line segment $RR'$,
and belong to
$
\bR_{>0}[z_i]+\bR_{>0}[z'_j]
\subset
H_1(T,\bR).
$
Note that $z_i$ and $z'_j$ intersect each other only once
on $M_\bR$ \cite[Lemma 7.1]{Ishii-Ueda_DMSMC}.
The other zigzag paths of $G$ are unchanged.
Therefore, the properly orderedness of $G$ implies that of $G'$,
and the zigzag polygon of $G'$ is $\Delta'$.
Since $G'$ is properly ordered,
the characteristic polygon of $G'$ coincides with the zigzag polygon $\Delta'$.
\end{proof}

We use Proposition \ref{pr:chop0} below
to remove the orbit of a corner:

\begin{proposition}\label{pr:chop0}
Let $G$ be a consistent symmetric dimer model
with characteristic polygon $\Delta$.
Let further $\frakc$ be a corner of $\Delta$, and
$\Delta'$ be the convex hull
of the complement
$(\Delta \cap N) \setminus H \frakc$
of the $H$-orbit of $\frakc$
in the set of lattice points of $\Delta$.
Assume that
for any $g \in H$,
the corners $\frakc$ and $g \frakc$ are not connected
by a primitive side segment of $\Delta$.
%Take the two primitive side segments $s_1$, $s_2$ of $\Delta_0$ connected to $\frakc$
%and assume
%that $\{s_1, s_2\} \cap \{gs_1, gs_2\} = \emptyset$ if $g \frakc \ne \frakc$.
Then there is a consistent symmetric dimer model $G'$
with characteristic polygon $\Delta'$.
\end{proposition}

\begin{proof}
Let $s_1$ and $s_2$ be the pair of primitive side segments of $\Delta$
incident to $\frakc$.
The assumption implies
that $\{s_1, s_2\} \cap \{gs_1, gs_2\} = \emptyset$ if $g \frakc \ne \frakc$.
Moreover, we have $g s_i \ne s_i$ for any non-trivial $g \in H$.

We use the operation in \cite[Section 10.1]{Ishii-Ueda_DMSMC}
for each corner
in the orbit of $\frakc$.
In \cite[Algorithm 10.1(1)]{Ishii-Ueda_DMSMC},
we take a pair $(z_1, z_2)$ of zigzag paths
corresponding to $\frakc$.
This means that the homology classes of $z_1$ and $z_2$ are normal
to $s_1$ and $s_2$ respectively.
Notice that although $s_i$ and $gs_i$ are different for $g\ne1$,
they might be contained in the same side of $\Delta$ and in that case
$z_i$ and $gz_i$ might coincide.
We claim that
by suitably choosing $z_i$,
we may assume $g z_i \ne z_i$ for any non-trivial $g \in H$.

Choose and fix a generic stability parameter $\theta$ invariant under $H$,
such as the $v_0$-generated stability for the fixed vertex $v_0$.
Then for each lattice point in $\Delta$,
there is a unique $\theta$-stable perfect matching
corresponding to it.
%For a pair of adjacent lattice points on the boundary of $\Delta$,
%the symmetric difference of the corresponding $\theta$-stable perfect matchings
%is a zigzag path whose slope is the normal vector to the side containing them.
%Let $z_i$ be the zigzag path obtained as the symmetric difference of the
%perfect matchings corresponding to the endpoints of $s_i$.
%Since $\theta$ is invariant, the action of $H$ on the set of perfect matchings
%preserves the $\theta$-stability and hence $g z_i$ corresponds to $gs_i$ for $g \in H$.
%Therefore, the assumption $g s_i \ne s_i$ implies $g z_i \ne z_i$.
The $\theta$-stable perfect matchings corresponding to boundary lattice points have the following property:
if $D$ and $D'$ are the $\theta$-stable perfect matchings corresponding to the endpoints of a primitive side segment $s$, then $D'$ is obtained from $D$ by ``flipping" a single zigzag path $z$ such that $[z]$ is outer normal
to the segment $s$ as in \cite[Corollary 3.8]{Gulotta}.
Indeed,
it follows from \cite[Corollary 3.8]{Gulotta} that
$D'$ is obtained from $D$ by flipping finitely many zigzag paths
$w_1, \dots, w_m$ with the same slope.
This means that
\begin{itemize}
\item
every other edge of $w_i$ belongs to $D$, and
\item
$D'$ is obtained from $D$
by replacing $D\cap w_i$ with $w_i \setminus D$
for all $i = 1, \ldots, m$.
\end{itemize}
If $m>1$, then since the height change $h(D, D')$ is a primitive vector, we can choose $w_1$ and $w_2$
so that their contributions to the height change cancel each other.
Notice that $T \setminus (w_1\cup w_2)$ has two connected components
and by our choice of $w_1$ and $w_2$, one connected component determines
submodules of $\bC\Gamma$-modules corresponding to the perfect matching $D$
and the same component determines quotient modules of $\bC\Gamma$-modules corresponding to $D'$.
This contradicts the $\theta$-stability of the perfect matchings $D$ and $D'$ and proves $m=1$.
Moreover, by fixing $\theta$, we obtain a bijective correspondence between the zigzag paths of $G$ and the primitive side segments of the characteristic polygon.
Let $z_i$ be the zigzag path corresponding to $s_i$ in this bijection.
Then $g(D_i)$ is obtained from $g(D_{\frakc})$ by flipping $g(z_i)$.
Since $\theta$ is invariant, the action of $H$ on the set of perfect matchings
preserves the $\theta$-stability and hence $g(D_{\frakc})$ and $g(D_i)$ are also $\theta$-stable.
This proves that $g(z_i)$ corresponds to $g(s_i)$ and $g(s_i) \ne s_i$ implies $g(z_i)\ne z_i$.

%%%%%%%
As in \cite[Algorithm 10.1(2)]{Ishii-Ueda_DMSMC},
we construct large hexagons from the pair $(z_1, z_2)$, and
identify them with vertices of the McKay quiver for a finite abelian
group $A \subset \GL(2, \bC) \subset \SL(3, \bC)$,
in such a way that
%Here the identification depends on the choice of
the large hexagon
corresponding to the trivial representation
contains the $H$-fixed face.
%and in our case, we choose the large hexagon containing the fixed face.
%Thus the choice is symmetric with respect to the $H$-action.
%%%%%%%%
Then we remove several edges on $z_1 \cap z_2$
as in \cite[Algorithm 10.1(3)]{Ishii-Ueda_DMSMC},
and for each $g\in H$, we do the same operation using the pair $(gz_1, gz_2)$.
If $g \frakc \ne \frakc$, then the assumption implies $\{z_1, z_2\} \cap \{gz_1, gz_2\} =\emptyset$
and hence the operations for $\{z_1, z_2\}$ and $\{gz_1, gz_2\}$ are independent.
If $g \frakc = \frakc$, then the action of $g$ exchanges $z_1$ and $z_2$,
preserving the edges to be removed.
Hence the consistent dimer model $G'$
obtained from $G$ by the successive operations
for the corners in the orbit of $\frakc$
is preserved by the action of $H$. 
The face of $G'$ containing the fixed face of $G$ is also fixed by $H$.
\end{proof}

\begin{example}
As an illustration of
Proposition \ref{pr:chop0},
consider the lattice triangle $\Delta$ and the dimer model $G$
having $\Delta$ as the characteristic polygon, shown in Figure \ref{fg:134_triangle}. Both $\Delta$ and $G$ are symmetric under $R_2$.

\begin{figure}[h]
\[
\begin{array}{ccc}
\begin{tikzpicture}[scale=0.75]
	\draw[fill=black] (0,0) circle   (1mm);
	\draw[fill=black] (-1,3) circle   (1mm);
	\draw[fill=black] (0,2) circle   (1mm);
	\draw[fill=black] (1,1) circle   (1mm);
	\draw[fill=black] (0,1) circle   (1mm);
	\draw[fill=black] (1,0) circle   (1mm);
	\draw[fill=black] (2,0) circle   (1mm);
	\draw[fill=black] (3,-1) circle   (1mm);	
	\draw[thick,-] (0,0) -- (-1,3) -- (3,-1) -- cycle;
	\node at (0,-1.5) {};	
\end{tikzpicture}
~~~~~~~~~~~&
\begin{tikzpicture}[scale=0.35,bend angle=45, looseness=1]
	\coordinate (A) at (0,0) {};
	\coordinate (B) at (12,0) {};
	\coordinate (C) at (0,12) {};	
	\coordinate (D) at (12,12) {};
	\coordinate (J1) at (1.2,0) {};	
	\coordinate (I1) at (4.29,0) {};	
	\coordinate (C2) at (7.71,0) {};	
	\coordinate (B2) at (10.8,0) {};	
	\coordinate (Z1) at (12,1.2) {};	
	\coordinate (W1) at (12,4.29) {};	
	\coordinate (V1) at (12,7.71) {};	
	\coordinate (U1) at (12,10.8) {};	
	\coordinate (T1) at (10.8,12) {};	
	\coordinate (S1) at (7.71,12) {};	
	\coordinate (R1) at (4.29,12) {};	
	\coordinate (Q1) at (1.2,12) {};	
	\coordinate (P1) at (0,10.8) {};	
	\coordinate (N1) at (0,7.71) {};	
	\coordinate (M1) at (0,4.29) {};	
	\coordinate (K1) at (0,1.2) {};	
	\coordinate (S) at (0.5,11.5) {};
	\coordinate (W) at (7,11) {};	
	\coordinate (A1) at (8,10) {};	
	\coordinate (T) at (2.5,9.5) {};	
	\coordinate (D2) at (3.5,8.5) {};	
	\coordinate (P) at (1,5) {};	
	\coordinate (O) at (2,4) {};	
	\coordinate (E2) at (5.5,6.5) {};	
	\coordinate (F2) at (6.5,5.5) {};	
	\coordinate (C1) at (10,8) {};	
	\coordinate (E1) at (11,7) {};	
	\coordinate (G2) at (8.5,3.5) {};	
	\coordinate (M) at (9.5,2.5) {};	
	\coordinate (L) at (4,2) {};	
	\coordinate (J) at (5,1) {};	
	\coordinate (K) at (11.5,0.5) {}; 
	\node at (0,-.25) {};	

	\draw[thick,-] (C) -- (S) -- (Q1) -- (S) -- (P1);
	\draw[thick,-] (N1) -- (T) -- (R1);
	\draw[thick,-] (M1) -- (P) -- (D2) -- (W) -- (S1);	
	\draw[thick,-] (K1) -- (O) -- (E2) -- (A1) -- (T1);
	\draw[thick,-] (J1) -- (L) -- (F2) -- (C1) -- (U1);
	\draw[thick,-] (I1) -- (J) -- (G2) -- (E1) -- (V1);
	\draw[thick,-] (C2) -- (M) -- (W1);
	\draw[thick,-] (B2) -- (K) -- (Z1) -- (K) -- (B);
	\draw[thick,-] (T) -- (D2);
	\draw[thick,-] (W) -- (A1);
	\draw[thick,-] (P) -- (O);
	\draw[thick,-] (E2) -- (F2);
	\draw[thick,-] (C1) -- (E1);
	\draw[thick,-] (L) -- (J);
	\draw[thick,-] (G2) -- (M);

	\draw[fill=black] (S) circle (2.5mm);
	\draw[fill=black] (D2) circle (2.5mm);
	\draw[fill=black] (O) circle (2.5mm);
	\draw[fill=black] (A1) circle (2.5mm);	
	\draw[fill=black] (J) circle (2.5mm);	
	\draw[fill=black] (E1) circle (2.5mm);	
	\draw[fill=black] (M) circle (2.5mm);
	\draw[fill=black] (F2) circle (2.5mm);

	\draw[fill=white] (T) circle (2.5mm);
	\draw[fill=white] (P) circle (2.5mm);
	\draw[fill=white] (W) circle (2.5mm);
	\draw[fill=white] (E2) circle (2.5mm);	
	\draw[fill=white] (L) circle (2.5mm);	
	\draw[fill=white] (C1) circle (2.5mm);	
	\draw[fill=white] (K) circle (2.5mm);
	\draw[fill=white] (G2) circle (2.5mm);
			
	\draw[color=gray] (A) -- (B) -- (12,12) --  (0,12) -- cycle;

%	\node at (\h+\l/2,0) {\gray 3};
%	\node at (0,-\h-\l/2) {\gray 2};
%	\node at (\h+\l/2,-\h-\l/2) {\gray 0};
\end{tikzpicture}

~~~~~~~~~~~&
\begin{tikzpicture}[scale=0.35,bend angle=45,>=latex]
% quiver part
	\coordinate (dA) at (0,0) {};
	\coordinate (dB) at (12,0) {};
	\coordinate (dC) at (0,12) {};	
	\coordinate (dD) at (12,12) {};
	\coordinate (dJ1) at (1.2,0) {};	
	\coordinate (dI1) at (4.29,0) {};	
	\coordinate (dC2) at (7.71,0) {};	
	\coordinate (dB2) at (10.8,0) {};	
	\coordinate (dZ1) at (12,1.2) {};	
	\coordinate (dW1) at (12,4.29) {};	
	\coordinate (dV1) at (12,7.71) {};	
	\coordinate (dU1) at (12,10.8) {};	
	\coordinate (dT1) at (10.8,12) {};	
	\coordinate (dS1) at (7.71,12) {};	
	\coordinate (dR1) at (4.29,12) {};	
	\coordinate (dQ1) at (1.2,12) {};	
	\coordinate (dP1) at (0,10.8) {};	
	\coordinate (dN1) at (0,7.71) {};	
	\coordinate (dM1) at (0,4.29) {};	
	\coordinate (dK1) at (0,1.2) {};	
	\coordinate (dS) at (0.5,11.5) {};
	\coordinate (dW) at (7,11) {};	
	\coordinate (dA1) at (8,10) {};	
	\coordinate (dT) at (2.5,9.5) {};	
	\coordinate (dD2) at (3.5,8.5) {};	
	\coordinate (dP) at (1,5) {};	
	\coordinate (dO) at (2,4) {};	
	\coordinate (dE2) at (5.5,6.5) {};	
	\coordinate (dF2) at (6.5,5.5) {};	
	\coordinate (dC1) at (10,8) {};	
	\coordinate (dE1) at (11,7) {};	
	\coordinate (dG2) at (8.5,3.5) {};	
	\coordinate (dM) at (9.5,2.5) {};	
	\coordinate (dL) at (4,2) {};	
	\coordinate (dJ) at (5,1) {};	
	\coordinate (dK) at (11.5,0.5) {}; 
% dimer nodes
	\node (A) at (0,0) {};
	\node (B) at (12,0) {};
	\node (C) at (0,12) {};	
	\node (D) at (12,12) {};
	\node (F) at (6,0) {6};	
	\node (M) at (12,6) {2};	
	\node (L1) at (6,12) {6};	
	\node (P) at (0,6) {2};	
	\node (N) at (3,3) {0};	
	\node (I) at (10.5,1.5) {7};	
	\node (L) at (7.5,4.5) {1};	
	\node (Q) at (4.5,7.5) {3};	
	\node (R) at (9,9) {4};	
	\node (V) at (1.5,10.5) {5};

% quiver N arrows
	\draw[red, thick,->] (V) -- (2,12);
	\draw[red, thick,->] (P) -- (V);
	\draw[red, thick,->] (Q) -- (L1);
	\draw[red, thick,->] (N) -- (Q);
	\draw[red, thick,->] (2,0) -- (N);
	\draw[red, thick,->] (R) -- (10,12);
	\draw[red, thick,->] (L) -- (R);
	\draw[red, thick,->] (F) -- (L);
	\draw[red, thick,->] (I) -- (M);
	\draw[red, thick,->] (10,0) -- (I);
% quiver NE arrows
	\draw[red, thick,->] (0,10) -- (V);
	\draw[red, thick,->] (V) -- (L1);
	\draw[red, thick,->] (P) -- (Q);
	\draw[red, thick,->] (Q) -- (R);
	\draw[red, thick,->] (R) -- (12,10);
	\draw[red, thick,->] (0,2) -- (N);
	\draw[red, thick,->] (N) -- (L);
	\draw[red, thick,->] (L) -- (M);
	\draw[red, thick,->] (F) -- (I);
	\draw[red, thick,->] (I) -- (12,2);
% quiver SW arrows
	\draw[red, thick,->] (3,12) -- (V);
	\draw[red, thick,->] (V) -- (0,9);
	\draw[red, thick,->] (L1) -- (P);
	\draw[red, thick,->] (9,12) -- (Q);
	\draw[red, thick,->] (Q) -- (0,3);
	\draw[red, thick,->] (12,12) -- (R);
	\draw[red, thick,->] (R) -- (N);
	\draw[red, thick,->] (N) -- (0,0);
	\draw[red, thick,->] (12,9) -- (L);
	\draw[red, thick,->] (L) -- (3,0);
	\draw[red, thick,->] (M) -- (F);
	\draw[red, thick,->] (12,3) -- (I);
	\draw[red, thick,->] (I) -- (9,0);

	\draw[color=gray] (0,0) -- (F) -- (12,0) -- (M) -- (12,12) --  (L1) -- (0,12) -- (P) -- (0,0);

% dimer edges
	\draw[gray,-] (dC) -- (dS) -- (dQ1) -- (dS) -- (dP1);
	\draw[gray,-] (dN1) -- (dT) -- (dR1);
	\draw[gray,-] (dM1) -- (dP) -- (dD2) -- (dW) -- (dS1);	
	\draw[gray,-] (dK1) -- (dO) -- (dE2) -- (dA1) -- (dT1);
	\draw[gray,-] (dJ1) -- (dL) -- (dF2) -- (dC1) -- (dU1);
	\draw[gray,-] (dI1) -- (dJ) -- (dG2) -- (dE1) -- (dV1);
	\draw[gray,-] (dC2) -- (dM) -- (dW1);
	\draw[gray,-] (dB2) -- (dK) -- (dZ1) -- (dK) -- (dB);
	\draw[gray,-] (dT) -- (dD2);
	\draw[gray,-] (dW) -- (dA1);
	\draw[gray,-] (dP) -- (dO);
	\draw[gray,-] (dE2) -- (dF2);
	\draw[gray,-] (dC1) -- (dE1);
	\draw[gray,-] (dL) -- (dJ);
	\draw[gray,-] (dG2) -- (dM);

% dimer vertex		
	\draw[fill=gray] (dS) circle (2.25mm);
	\draw[fill=gray] (dD2) circle (2.25mm);
	\draw[fill=gray] (dO) circle (2.25mm);
	\draw[fill=gray] (dA1) circle (2.25mm);	
	\draw[fill=gray] (dJ) circle (2.25mm);	
	\draw[fill=gray] (dE1) circle (2.25mm);	
	\draw[fill=gray] (dM) circle (2.25mm);
	\draw[fill=gray] (dF2) circle (2.25mm);

	\draw[gray, fill=white] (dT) circle (2.25mm);
	\draw[gray, fill=white] (dP) circle (2.25mm);
	\draw[gray, fill=white] (dW) circle (2.25mm);
	\draw[gray, fill=white] (dE2) circle (2.25mm);	
	\draw[gray, fill=white] (dL) circle (2.25mm);	
	\draw[gray, fill=white] (dC1) circle (2.25mm);	
	\draw[gray, fill=white] (dK) circle (2.25mm);
	\draw[gray, fill=white] (dG2) circle (2.25mm);
\end{tikzpicture}

\end{array}
\]
   \caption{Lattice triangle $\Delta$ and dimer model $G$ both symmetric under $R_2$, and the McKay quiver of $\frac{1}{8}(1,3,4)$ dual to $G$.}
\label{fg:134_triangle}
\end{figure}

The affine toric variety $X_\Delta$ is isomorphic
to the quotient of the affine space $\bC^3$
by the cyclic subgroup
$\la \frac{1}{8}(1,3,4) \ra \subset \SL_3(\bC)$
of order 8
generated by
$\diag(\zeta, \zeta^3, \zeta^4)$,
where $\zeta$ is a primitive 8-th root of unity.
In general,
given a subgroup $\la \frac{1}{n}(1,q,n-q-1) \ra$ of $\SL_3(\bC)$,
we define integers $r, b_1, \ldots, b_r, i_0, \ldots, i_{r+1}$ by
$i_0 \coloneqq n$,
$i_1 \coloneqq q$,
$
i_t = b_{t+1} i_{t+1} - i_{t+2}
$
(where $0 < i_{t+2} < i_{t+1}$),
$i_r = 1$, and $i_{r+1} = 0$
as explained in \cite[Section 4]{Ishii-Ueda_DMSMC}
(which goes back to \cite{Wunram2,Wunram}).
For $n=8$ and $q=3$,
we have
\begin{align}
 8 = 3 \cdot 3 - 1,
\end{align}
so that $r=2$ and
$(i_0, i_1, i_2) = (8, 3, 1)$.
By removing the edges of $G$
dual to the arrows of the McKay quiver
corresponding to `multiplication by $z$'
(i.e., those which goes in the southwest direction
in the quiver of Figure \ref{fg:134_triangle})
from the vertices $i_0$, $i_1$, $i_2$
and removing divalent nodes,
one obtains the dimer model $G'$
shown in Figure \ref{fg:trapezoid_dimer},
whose characteristic polygon is the trapezoid
shown in the same figure.

\begin{figure}[h]
\[
\begin{array}{cc}
\begin{tikzpicture}[scale=0.75]
%	\draw[fill=black] (0,0) circle   (1mm);
	\draw[fill=black] (-1,3) circle   (1mm);
	\draw[fill=black] (0,2) circle   (1mm);
	\draw[fill=black] (1,1) circle   (1mm);
	\draw[fill=black] (0,1) circle   (1mm);
	\draw[fill=black] (1,0) circle   (1mm);
	\draw[fill=black] (2,0) circle   (1mm);
	\draw[fill=black] (3,-1) circle   (1mm);	
	\draw[thick,-] (0,1) -- (-1,3) -- (3,-1) -- (1,0) -- cycle;
	\node at (0,-1.5) {};	
\end{tikzpicture}
~~~~~~~~~~~~~~~~&
\begin{tikzpicture}[scale=0.35,bend angle=45, looseness=1]

	\draw[thick,-] (0,10) -- (2,12);
	\draw[thick,-] (0,8) -- (4,12);
	\draw[thick,-] (0,4) -- (8,12);	
	\draw[thick,-] (0,2) -- (10,12);
	\draw[thick,-] (2,0) -- (12,10);
	\draw[thick,-] (4,0) -- (12,8);
	\draw[thick,-] (8,0) -- (12,4);
	\draw[thick,-] (10,0) -- (12,2);
	\draw[thick,-] (2,10) -- (4,8);
	\draw[thick,-] (7,11) -- (8,10);
	\draw[thick,-] (5,7) -- (7,5);
	\draw[thick,-] (10,8) -- (11,7);
	\draw[thick,-] (8,4) -- (10,2);

	\draw[fill=black] (4,8) circle (2.5mm);
	\draw[fill=black] (8,10) circle (2.5mm);
	\draw[fill=black] (7,5) circle (2.5mm);
	\draw[fill=black] (11,7) circle (2.5mm);	
	\draw[fill=black] (10,2) circle (2.5mm);	

	\draw[fill=white] (2,10) circle (2.5mm);
	\draw[fill=white] (7,11) circle (2.5mm);
	\draw[fill=white] (5,7) circle (2.5mm);
	\draw[fill=white] (10,8) circle (2.5mm);	
	\draw[fill=white] (8,4) circle (2.5mm);	
			
	\draw[color=gray] (0,0) -- (12,0) -- (12,12) --  (0,12) -- cycle;

%	\node at (\h+\l/2,0) {\gray 3};
%	\node at (0,-\h-\l/2) {\gray 2};
%	\node at (\h+\l/2,-\h-\l/2) {\gray 0};
\end{tikzpicture}
\end{array}
\]
\caption{The dimer model $G'$ also symmetric with respect to $R_2$.}
\label{fg:trapezoid_dimer}
\end{figure}

\end{example}

\subsection{Cyclic groups}
 \label{sc:cyclic}

In this section,
we assume that $H$ is a cyclic group of order $n$
consisting of rotations.
In this case, Proposition \ref{pr:chop0} implies the following:

\begin{corollary} \label{cr:chop1}
Let $G$ be a consistent symmetric dimer model
with characteristic polygon $\Delta$.
Let further $\frakc$ be a corner of $\Delta$ and
$\Delta'$ be the lattice polygon
obtained from $\Delta$
by removing the orbit of $\frakc$.
Assume that one of the following holds:
\begin{enumerate}[(1)]
 \item \label{it:non_n-gon}
$\Delta$ is not an $n$-gon.
 \item \label{it:n-gon}
$\Delta$ is an $n$-gon with a
boundary lattice point which is not a corner.
\end{enumerate}
Then there is a consistent symmetric dimer model $G'$
with characteristic polygon $\Delta'$.
\end{corollary}

\subsubsection{The group $C_2$}

In this case,
we can embed $\Delta$ in a square $\Deltatilde$
and iterate the operations in Corollary \ref{cr:chop1},
since Condition \eqref{it:non_n-gon} in \ref{cr:chop1}
always holds for $n=2$.

%\begin{comment}
%
%\[
%\begin{tikzpicture}[scale=1.2]
%% Draw big square  	
%  	\draw[thick] (1,-1) -- (1,1) -- (-1,1) -- (-1,-1) -- cycle;
%  	\draw[fill=black] (1,-1) circle (0.75mm);
%  	\draw[fill=black] (1,1) circle (0.75mm);
%  	\draw[fill=black] (-1,1) circle (0.75mm);
%  	\draw[fill=black] (-1,-1) circle (0.75mm);
%  	\draw (1,-1) node [below right] {$P$};
%  	\draw (-1,1) node [above left] {$gP$};
%  	
%% Draw interior hexagon  	
%  	\draw[thick] (0.5,0) -- (0,0.5) -- (-0.5,0.5) -- (-0.5,0) -- (0,-0.5) -- (0.5,-0.5) -- cycle;
%  	\draw[fill=black] (0.5,0) circle (0.75mm);
%  	\draw[fill=black] (0,0.5) circle (0.75mm);
%  	\draw[fill=black] (-0.5,0.5) circle (0.75mm);
%  	\draw[fill=black] (-0.5,0) circle (0.75mm);
% 	\draw[fill=black] (0,-0.5) circle (0.75mm);
%  	\draw[fill=black] (0.5,-0.5) circle (0.75mm);
%
%% Draw zigzags
%	\draw[->,red] (1.2,-0.5) -- (1.8,-0.5);\draw[red] (1.5,-0.3) node {$z_1$};
%	\draw[->,red] (0.5,-1.2) -- (0.5,-1.8);\draw[red] (0.2,-1.5) node {$z_2$};
%	\draw[->,red] (-1.2,0.5) -- (-1.8,0.5);\draw[red] (-1.5,0.2) node {$gz_1$};
%	\draw[->,red] (-0.5,1.2) -- (-0.5,1.8);\draw[red] (-0.2,1.5) node {$gz_2$};
%\end{tikzpicture}
%\]
%
%\end{comment}

\subsubsection{The group $C_3$}
 \label{sc:C3}

Let $\Delta_n$ be the convex hull of
$(n, -n)$, $(n, 2n)$ and $(-2n, -n)$,
which is the characteristic polygon
of the hexagonal dimer model $G_n$
associated with the McKay quiver
of the abelian subgroup $A$ of $\SL(3,\bC)$
isomorphic to
$
 \bZ/3n\bZ \times \bZ/3n\bZ.
$
By translating $G_n$ if necessary,
we assume that the face
corresponding to the trivial representation of $A$
is fixed by the action of $H$.
For a symmetric lattice polygon $\Delta$,
take the minimum integer $n$ such that
$
 \Delta \subset \Delta_n
$
and put
$
 \Deltatilde \coloneqq \Delta_n.
$
Then we have
$
 \partial \Deltatilde \cap \Delta \ne \emptyset.
$
By starting from $\Gtilde \coloneqq G_n$
and iterate the operations in Corollary \ref{cr:chop1},
we obtain a consistent symmetric dimer model.

\begin{remark}
For a lattice polygon $\Delta$ with rotational symmetry of order $3$
whose center is not a lattice point
(in this case $C_3 \subset \GL(N) \ltimes N$ but $C_3 \not\subset \GL(N)$),
we can embed $\Delta$ into a lattice polygon
corresponding to the Abelian subgroup of $\SL(3, \bC)$
isomorphic to $\bZ/2n\bZ \times \bZ/2n\bZ$,
and the same method
produces a consistent symmetric dimer model. This includes in our treatment the case when $X_\Delta/H\cong\bC^3/G$ where $G$ is a trihedral group in $\SL(3,\bC)$.
\end{remark}

\subsubsection{The group $C_4$}

Let $\Delta_n$ be the convex hull
of $(\pm n, 0)$ and $(0, \pm n)$.
A dimer model $G_n$
with characteristic polygon $\Delta_n$
can be obtained from
the consistent dimer model
with characteristic polygon $\Delta_1$
shown in Figure \ref{fg:Rotation4}
by using Lemma
%changing the fundamental region as in
\ref{lm:UY}.
This dimer model is symmetric
with respect to the action of the group $C_4$
fixing an octogonal face.
Note that the face of a dimer model
symmetric under a rotation of order $4$
must have at least 8 edges.

Given a $C_4$-invariant lattice polygon $\Delta$,
we embed it into $\Delta_n$ with the smallest $n$,
and iterate the operations in Corollary \ref{cr:chop1}
to obtain a consistent symmetric dimer model
with characteristic polygon $\Delta$.

\begin{figure}[h]
\[
\begin{array}{cc}
\begin{tikzpicture}[scale=0.75]
	\draw[fill=black] (0,0) circle   (1mm);
	\draw[fill=black] (-1,1) circle   (1mm);
	\draw[fill=black] (0,2) circle   (1mm);
	\draw[fill=black] (1,1) circle   (1mm);
	\draw[fill=black] (0,1) circle   (1mm);
	\draw[thick,-] (0,0) -- (-1,1) -- (0,2) -- (1,1) -- cycle;
\end{tikzpicture}
~~~~~~~~~~~~~~~~&
\begin{tikzpicture}[scale=1.2,bend angle=45, looseness=1]
	\def\h{0.9238} %height of the octagon
	\def\l{0.76537} %length of side of the octagon = side of square
	\def\c{0.5412} %length of side of the octagon = side of square
	
	\draw[color=gray] (-\l/2-\c/2,\l/2+\c/2) -- (\h+\l+\c/2,\l/2+\c/2) -- (\h+\l+\c/2,-\h-\l-\c/2) --  (-\l/2-\c/2,-\h-\l-\c/2) -- cycle;

	\draw[thick]  (-\l/2-\c/2,-\h-\l-\c/2) -- (-\l/2,-\h-\l) -- (\l/2,-\h-\l) -- (\l/2+\c/2,-\h-\l-\c/2); 
	\draw[thick] (-\l/2-\c/2,-\l/2-\c/2) -- (-\l/2,-\h) -- (-\l/2,-\h-\l);
	\draw[thick] (-\l/2,-\h) -- (\l/2,-\h) -- (\l/2,-\h-\l);
	\draw[thick] (\l/2,-\h) -- (\h,-\l/2) -- (\h,\l/2) -- (\h-\c/2,\l/2+\c/2);
	\draw[thick] (\h,-\l/2) -- (\h+\l,-\l/2) -- (\h+\l+\c/2,-\l/2-\c/2);
	\draw[thick] (\h+\l,-\l/2) -- (\h+\l,\l/2) -- (\h,\l/2) -- (\h+\l,\l/2) -- (\h+\l+\c/2,\l/2+\c/2);

	\draw[fill=white] (-\l/2,-\h) circle (1mm);
	\draw[fill=black] (\l/2,-\h) circle (1mm);
	\draw[fill=black] (-\l/2,-\h-\l) circle (1mm);
	\draw[fill=white] (\l/2,-\h-\l) circle (1mm);

	\draw[fill=white] (\h,-\l/2) circle (1mm);
	\draw[fill=black] (\h,\l/2) circle (1mm);
	\draw[fill=black] (\h+\l,-\l/2) circle (1mm);
	\draw[fill=white] (\h+\l,\l/2) circle (1mm);

%	\node at (0,0) {\gray 1};
%	\node at (\h+\l/2,0) {\gray 3};
%	\node at (0,-\h-\l/2) {\gray 2};
%	\node at (\h+\l/2,-\h-\l/2) {\gray 0};
\end{tikzpicture}
\end{array}
\]
   \caption{Lattice polygon $\Delta_1$ and the dimer model $G_1$.}
\label{fg:Rotation4}
\end{figure}  

\subsubsection{The group $C_6$}

Let $G_1$ be the dimer model
%shown in \pref{fg:Hexagon}
with characteristic polygon $\Delta_1$
shown in Figure \ref{fg:Hexagon}.
The dimer model $G_n$
with characteristic polygon $\Delta_n \coloneqq n \Delta_1$
is obtained as the $\bZ/n\bZ \times \bZ/n\bZ$-cover of $G_1$
by using Lemma \ref{lm:UY}
as in previous cases.
Given a $C_6$-invariant lattice polygon $\Delta$,
we embed it into $\Delta_n$ with the smallest $n$,
and iterate the operations in Corollary \ref{cr:chop1}
to obtain a consistent symmetric dimer model
with characteristic polygon $\Delta$.

%Note that it is the dimer model of Type $\ast632$ in Figure \ref{dimer17}. 

\begin{figure}[h]
\[
\begin{array}{cc}
\begin{tikzpicture}[scale=0.75]
	\def\hexdot#1#2{
  	\draw[thick] #1 #2 +(0:1) \foreach \a in {60,120,180,240,300} { -- +(\a:1) } -- cycle;
 	\draw[fill=black] #1 #2 \foreach \a in {0,60,120,180,240,300} { +(\a:1) circle (1mm) } ;
 	}
	\hexdot{(0,0)}{(0,0)}
	\draw[fill=black] (0,0) circle   (1mm);
\end{tikzpicture}

~~~~~~~~~~~~~&

\begin{tikzpicture}[scale=0.75]
% Definition of hexagon (thick)
\def\hext#1#2{
  	\draw[thick] #1 #2 +(0:1) \foreach \a in {60,120,180,240,300} { -- +(\a:1) } -- cycle;
 	\draw[fill=white] #1 #2 \foreach \a in {0,120,240} { +(\a:1) circle (1.5mm) } ;
  	\draw[fill=black] #1 #2 \foreach \a in {60,180,300} { +(\a:1) circle (1.5mm) } ;
	}
% Definition of hexagon (thick and reverse colors)
\def\hexopt#1#2{
  	\draw[thick] #1 #2 +(0:1) \foreach \a in {60,120,180,240,300} { -- +(\a:1) } -- cycle;
 	\draw[fill=black] #1 #2 \foreach \a in {0,120,240} { +(\a:1) circle (1.5mm) } ;
  	\draw[fill=white] #1 #2 \foreach \a in {60,180,300} { +(\a:1) circle (1.5mm) } ;
	}
% Rest of the lines	
	\draw[thick] (1,0) -- (1+\hh,0.5);\draw[thick] (60:1) -- (0.5+\hh,0.5+\hh);
	\draw[thick] (-0.5,-\hh) -- (-0.5,-0.5-\hh);\draw[thick] (0.5,-\hh) -- (0.5,-0.5-\hh);
	\draw[thick] (1+\hh,0.5+2*\hh) -- (1+\hh,1+2*\hh);\draw[thick] (2+\hh,0.5+2*\hh) -- (2+\hh,1+2*\hh);
	\draw[thick] (-1,0) -- (-1-0.5*\hh,0.25);\draw[thick] (-0.5,\hh) -- (-0.5-0.5*\hh,\hh+0.25);
	\draw[thick] (2+\hh,0.5) -- (2+\hh+0.5*\hh,0.5-0.25);\draw[thick] (2.5+\hh,0.5+\hh) -- (2.5+\hh+0.5*\hh,0.5+\hh-0.25);
% hexagons	
	\hexopt{(0,0)}{(0,0)}
	\hext{(0,0)}{(30:2*\hh+1)}
% for the fundamental domain
	\draw[color=gray] (-1.5-\hh,-0.5-\hh) -- (0,1+2*\hh) -- (3+2*\hh,1+2*\hh) -- (1.5+\hh,-0.5-\hh) -- cycle;
% Numbers
%	\node at (-1.5-\hh,-0.5-\hh) {$\red 0$};\node at (0,1+2*\hh) {$\red 0$};
%	\node at (3+2*\hh,1+2*\hh) {$\red 0$};\node at (1.5+\hh,-0.5-\hh) {$\red 0$};
%	\node at (0,0) {$\red 1$};\node at (30:2*\hh+1) {$\red 2$};\node at (30:\hh+0.5) {$\red 5$};
%	\node at (0,-\hh-0.5) {$\red 4$};\node at (1.5+\hh,1+2*\hh) {$\red 4$};
%	\node at (-1.2,0.65) {$\red 3$};\node at (2+\hh+0.66,0.65) {$\red 3$};
\end{tikzpicture}
%&
%\begin{tikzpicture}[bend angle=120]
%\node[name=s,regular polygon, regular polygon sides=5, minimum size=3cm] at (0,0) {}; 
%\node (1) at (0,0) {$\red 0$};
%\node (6) at (s.corner 1)  {$\red 5$};
%\node (2) at (s.corner 2)  {$\red 1$};
%\node (4) at (s.corner 3)  {$\red 3$};
%\node (5) at (s.corner 4)  {$\red 4$};
%\node (3) at (s.corner 5)  {$\red 2$};
% x arrows
%\draw[->]  (1)+(105:6.5pt) -- ($(6)+(-105:6.5pt)$);
%\draw[->]  (1)+(75:6.5pt) -- ($(6)+(-75:6.5pt)$);
%\draw[->]  (1)+(219:6.5pt) --   ($(4)+(70:6.5pt)$);
%\draw[->]  (1)+(249:6.5pt) --   ($(4)+(40:6.5pt)$);
%\draw[->]  (1)+(291:6.5pt) -- ($(5)+(140:6.5pt)$);
%\draw[->]  (1)+(319:6.5pt) --   ($(5)+(110:6.5pt)$);
%
%\draw[->]  (2)+(0:6.5pt) -- ($(1)+(140:6.5pt)$);
%\draw[->]  (2)+(-20:6.5pt) --   ($(1)+(160:6.5pt)$);
%\draw[->]  (2)+(-40:6.5pt) --   ($(1)+(180:6.5pt)$);
%\draw[->]  (3)+(180:6.5pt) -- ($(1)+(40:6.5pt)$);
%\draw[->]  (3)+(200:6.5pt) --   ($(1)+(20:6.5pt)$);
%\draw[->]  (3)+(220:6.5pt) --   ($(1)+(00:6.5pt)$);
%
%\draw[->]  (6) -- (2);
%\draw[->]  (4) -- (2);
%\draw[->]  (6) -- (3);
%\draw[->]  (5) -- (3);
%
%\draw[->] (4) .. controls (1,-2) and (2,-2) .. (3);
%\draw[->] (5) .. controls (-1,-2) and (-2,-2) .. (2);
%\end{tikzpicture}

\end{array}
\]
\caption{The hexagon $\Delta_1$ with one interior lattice point and the dimer model $G_1$.}
\label{fg:Hexagon}
\end{figure}

\subsection{Reflection groups of order two}
 \label{sc:reflection}

In the case of reflection groups,
we take the square lattice dimer model $\Gtilde$
whose characteristic polygon $\Deltatilde$ is a rectangle
as $\Gtilde$.
%(such a dimer model is well-known and can be found in \cite[Figure 4]{Hanany-Vegh}).
%Notice both $R_1$ and $R_2$ act on $\Gtilde_1$ although the associate action
%of $R_1$ on $\Deltatilde_1$ does not have a fixed lattice point.

\subsubsection{The group $R_1$}

For an $R_1$-invariant lattice polygon $\Delta$,
let $\Deltatilde$ be the minimum rectangle containing $\Delta$,
whose sides are parallel to $(1,0)$ or $(0,1)$,
i.e., two of whose sides are parallel to $(1,0)$,
and the other two are parallel to $(0,1)$.
%Then $\Deltatilde$ is a lattice rectangle and
%there is an integer matrix $P$ which sends $\Deltatilde_1$ to $\Deltatilde$
%and therefore we obtain a consistent dimer model $\Gtilde$ with $R_1$ action
%corresponding to $\Deltatilde$.
Since each side of $\Deltatilde$ contains a lattice point of $\Delta$,
we can start from the square lattice dimer model $\Gtilde$ and
iterate the operations in Proposition \ref{pr:chop0}
to obtain a consistent symmetric dimer model $G$
with characteristic polygon $\Delta$.

\subsubsection{The group $R_2$}

In this case,
we consider a rectangle containing $\Delta$,
two of whose sides are parallel either to $(1,1)$ or $(1, -1)$.
Note that if we require that each of the four sides meet $\Delta$,
then the rectangle may not be a lattice rectangle,
i.e.,
it may not have lattice points as its corners.
In general, there may be two minimal such lattice rectangles containing $\Delta$.
We choose $\Deltatilde$
such that $\partial \Deltatilde$ contains
$\partial \Delta \cap \bZ(1,1)$ (if this is non-empty).
Then we can again iterate the operations in Proposition \ref{pr:chop0} to obtain a dimer model $G$
corresponding to $\Delta$.

\subsection{Dihedral groups}
 \label{sc:dihedral}

\subsubsection{The group $D_4^1$}

For a lattice polygon $\Delta$ symmetric under the $D_4^1$-action,
let $\Deltatilde$ be the minimum rectangle containing $\Delta$
whose sides are parallel to $(1,0)$ or $(0,1)$.
Then starting from $\Deltatilde$,
we can iterate the operations in Proposition \ref{pr:chop0}
to obtain a consistent symmetric dimer model $G$
with characteristic polygon $\Delta$.

\subsubsection{The group $D_4^2$}
Consider the action of $D_4^2$ on $\bR^2$ and
let $L_1= \bR(1,1)$ and $L_2=\bR(1,-1)$ be the
lines of reflections.
Then $D_4^2$ acts freely on $\bR^2 \setminus (L_1 \cup L_2)$.
We use rectangles as in the $R_2$ case.

%\begin{lemma}
%Let $\Deltatilde$ be a $D_4^2$-invariant lattice rectangle whose sides are parallel
%to $L_1$ or $L_2$.
%Then there exists a consistent dimer model with $D_4^2$-action
%which has a fixed face.
%\end{lemma}
%\begin{proof}
%Such a rectangle is a linear transform of the lattice square whose lattice points
%are the corners and its center of gravity, i.e., $\Delta_0$ in \pref{fg:Rotation4}, 
%which corresponds to the dimer
%model $\Gmin$ in \pref{fg:Rotation4}.
%Hence a corresponding dimer model is obtained from $\Gmin$ by
%enlarging the fundamental region of the torus as in \cite{Ueda-Yamazaki_NBTMQ}.
%\end{proof}

\begin{lemma}
Let $\Deltatilde$ be a $D_4^2$-invariant lattice rectangle whose sides are parallel
to $L_1$ or $L_2$.
Then the number of lattice points on  $\partial \Deltatilde \cap (L_1 \cup L_2)$ is
either $0$ or $4$.
\end{lemma}
\begin{proof}
Let $v_1$ and $v_2$ be points on $\partial \Deltatilde \cap L_1$ and $\partial \Deltatilde \cap L_2$
respectively.
Then one has $\partial \Deltatilde \cap (L_1 \cup L_2) = \{ \pm v_1, \pm v_2\}$, and
$\pm v_1 \pm v_2$ are the four corners of $\Deltatilde$, which are lattice points.
The assertion follows from this.
\end{proof}

Let $\Delta$ be a lattice polygon invariant under the $D_4^2$-action.
We embed $\Delta$ into an invariant lattice rectangle $\Deltatilde$
whose sides are parallel to $L_1$ or $L_2$.
We assume that all the lattice points on
$\partial \Delta \cap (L_1 \cup L_2)$ are on $\partial \Deltatilde$
and that $\Deltatilde$ is the minimum of the lattice rectangles satisfying this condition.
This means the following.
\begin{itemize}
\item
If $\#(\partial \Delta \cap (L_1 \cup L_2) \cap \bZ^2)=0$ or $4$, then
$\partial \Delta \cap (L_1 \cup L_2)=\partial \Deltatilde \cap (L_1 \cup L_2)$.
\item
Suppose $\partial \Delta \cap L_i \cap \bZ^2 \ne \emptyset$ and
$\partial \Delta \cap L_j \cap \bZ^2 = \emptyset$ for $\{i,j\}=\{1,2\}$.
Then $\partial \Deltatilde \cap L_i$ coincides with $\partial \Delta \cap L_i$,
while $\partial \Deltatilde \cap L_j$ consists of the lattice points
closest to $\partial \Delta \cap L_j$ outside of $\Delta$.
\end{itemize}
Then starting from $\Deltatilde$,
we can iterate the operations in Proposition \ref{pr:chop0}
to obtain a consistent symmetric dimer model $\Delta$.

\subsubsection{The group $D_6^1$}
Let $\Delta$ be a $D_6^1$-invariant lattice polygon $\Delta$.
As in Section \ref{sc:C3},
take the minimum integer $n$ such that $\Delta \subset \Deltatilde_n$,
where $\Deltatilde_n$ is the convex hull of $(n, -n)$, $(n, 2n)$ and $(-2n, -n)$.
In this case, we cannot obtain $\Delta$ from $\Deltatilde_n$
by iteration of chopping corners
satisfying the conditions in Proposition \ref{pr:chop0}
if at some step the corner is on a primitive side segment intersecting a line of reflection. 
Thus before applying Proposition \ref{pr:chop0}, we first cut off regular triangles at the corners of $\Deltatilde_n$;
let $\Deltatilde$ be the minimum hexagon containing $\Delta$ obtained by cutting off three corner regular triangles from $\Deltatilde_n$ (when $\Delta$ is a triangle, we obtain $\Delta$ itself instead of a hexagon but in this case, there is nothing to do).
We apply Proposition \ref{pr:Gulotta} simultaneously to the three corners of $\Deltatilde_n$
by symmetrically choosing the zigzag paths in Proposition \ref{pr:Gulotta}.
This operation produces a symmetric consistent dimer model $\Gtilde$
whose characteristic polygon is $\Deltatilde$.
To obtain $\Delta$ from $\Deltatilde$, 
notice that the minimality of the hexagon $\Deltatilde$ ensures
that $\partial\Deltatilde$ contains all the points of $\partial \Delta$ that are on the lines of reflections.
%that $\Delta$
%contains the middle points of the sides of $\Deltatilde$
%that are on the lines of reflections.
Therefore,
for any corner $\frakc$ of $\Deltatilde$ which is not on $\Delta$,
$g\frakc$ and $\frakc$ are not connected by a primitive line segment of $\Deltatilde$ for a non-trivial $g \in D^1_6$.
Thus we can iterate the operations in Proposition \ref{pr:chop0} to obtain a consistent symmetric dimer model corresponding to $\Delta$.

\subsubsection{The group $D_6^2$}
We fix a $D_6^2$-invariant metric on $\bR^2$ such that $(1,0)$ is of length $1$.
This means that we consider the inner product defined by
$$
\langle (x_1, y_1), (x_2, y_2) \rangle = \begin{pmatrix} x_1 & y_1 \end{pmatrix} \begin{pmatrix} 1 & -\frac{1}{2} \\ -\frac{1}{2} & 1 \end{pmatrix}
\begin{pmatrix} x_2 \\ y_2 \end{pmatrix}.
$$
For example, $(1,0)$ is perpendicular to $(1, 2)$.

In this case, the lines of reflections are
$L_1\coloneqq \bR(1,0)$, $L_2\coloneqq\bR(1,1)$ and $L_3\coloneqq\bR(0,1)$.
Let $\Deltatilde_n$ be the lattice hexagon whose corners are
$(n, 0)$, $(n,n)$, $(0, n)$, $(-n, 0)$ $(-n,-n)$ and $(0,-n)$,
which is a regular hexagon of side $n$.
Then $\Deltatilde_1$ is in Figure \ref{fg:Hexagon} and thus
 a consistent dimer model $\Gtilde_n$ corresponding to $\Deltatilde_n$ with $D_6^2$-action
 is obtained by
 applying Lemma \ref{lm:UY} to the one in Figure \ref{fg:Hexagon}.

For a given lattice polygon $\Delta$ with $D_6^2$-action,
embed $\Delta$ into $\Deltatilde_n$ with the minimum value of $n$.
This means $\partial \Delta \cap \partial \Deltatilde_n \neq \emptyset$.
We first cut off isosceles triangles from $\Deltatilde_n$ as follows.
Let $k$ and $l$ be the maximum integers satisfying $(k,0) \in \Delta$ and $(l,l) \in \Delta$
respectively.
Notice that $(k,0)$ is on $L_1$  and $(l,l)$ is on $L_2$.
Let $\Deltatilde$ be the convex lattice polygon obtained by cutting off corner triangles of $\Deltatilde_n$
by the following six lines:
\begin{itemize}
\item
the lines passing through $(k,0)$ or $(-l,0)$ and  perpendicular to $L_1$,
\item
the lines passing through $(l,l)$ or $(-k,-k)$ and perpendicular to $L_2$,
\item
the lines passing through $(0,k)$ or $(0,-l)$ and perpendicular to $L_3$. 
\end{itemize}
Since $\Delta$ is convex and invariant by $D_6^2$, it is contained in $\Deltatilde$.
Moreover, $\Delta$ contains the intersections of the lines of reflections with $\partial \Deltatilde$.
%Then the assumption that $\partial \Deltatilde_n \cap \Delta \ne \emptyset$
%implies $2k+2l \ge 3n$.
By applying Proposition \ref{pr:Gulotta} at the six corners in a symmetric way,
we obtain a symmetric consistent dimer model with $D_6^2$-action and a fixed face
corresponding to $\Deltatilde$.
To obtain $\Delta$ from $\Deltatilde$, we iterate the operation of chopping corners in a $D_6^2$-orbit.
In this process, by our choice of $\Deltatilde$, a corner $\frakc$ and $g \frakc$ are not adjacent to each other for a non-trivial $g \in D_6^2$.
Therefore we can apply Proposition \ref{pr:chop0} in each step.
Thus there is a consistent dimer model with $D_6^2$-action and a fixed face whose characteristic polygon is $\Delta$.

\subsubsection{The group $D_8$}
Let $G_n$ be the dimer model
which corresponds to the square $\Delta_n$
as in the $C_4$ case.
Then we have an action of $D_8$ on $G_n$.
Take the smallest $\Delta_n$ containing $\Delta$
and cut off four isosceles triangles from
the corners such that 
\begin{itemize}
\item
the resulting polygon (octagon in general)
$\Deltatilde$ contains $\Delta$ and
\item
$\Deltatilde$ is the minimum of such polygons.
\end{itemize}
Then again by applying Proposition \ref{pr:Gulotta} to the four corners of $\Delta_n$
in a symmetric way, we obtain a symmetric consistent dimer model $\Gtilde$
with characteristic polygon $\Deltatilde$. 
Now $\Delta$ can be obtained from $\Deltatilde$
by iteration of chopping corners as in Proposition \ref{pr:chop0}.
Thus we obtain a desired dimer model corresponding to $\Delta$.

\subsubsection{The group $D_{12}$}
In this case let $\Deltatilde$ be the minimum polygon obtained by cutting corner triangles of
the hexagon $\Deltatilde_n$ exactly as in the $D_6^2$ case.
(Notice that we have $k=l$ in the $D_{12}$ case.)
Then the same argument as in the $D_6^2$ case proves the existence of a consistent dimer model with $D_{12}$-action
corresponding to $\Delta$.

\section{Non-commutative crepant resolutions}
 \label{sc:NCCR}

%We prove \pref{th:NCCR} in this section.
%Assume that from every $H$-invariant $\Delta$ we can construct a consistent $H$-invariant dimer model $G_\Delta$ in every case. 
Let $G$ be a consistent dimer model
with characteristic polygon $\Delta$
and
$\Gamma$ be the corresponding quiver with relations.
As we recalled in Section \ref{sc:moduli},
the moduli space $\cM_\theta$
of stable representations of $\Gamma$
with respect to a generic stability parameter $\theta$
is a crepant resolution
$
 \tau \colon \cM_\theta \to X_\Delta
$
of the Gorenstein affine toric variety
$
 X_\Delta = \Spec R,
$
and the tautological bundle
$
 \cE \coloneqq \bigoplus_{v \in Q_0} \cL_v
$
is a tilting bundle
such that
$
 \End(\cE) \cong \bC\Gamma.
$
Fix a vertex $v_0 \in Q_0$.
By replacing $\cE$ with $\cE \otimes \cL_{v_0}^{-1}$ if necessary,
we may assume
$
 \cL_{v_0}\cong\cO_{\cM_\theta}.
$
Then \cite[Proposition A.2]{Toda-Uehara} shows
that
$\End (\cE)$
is isomorphic to the endomorphism algebra
$
 \End_R(E)
%  \cong \bC\Gamma
$
of the $R$-module
$
 E
  \coloneqq \tau_* \cE
  \cong H^0(\cE),
$
and that
$\End_R(E)$ is a non-commutative crepant resolution of $R$
in the sense of \cite{MR2077594}.

Let $G$ be a dimer model
which is symmetric with respect to the action of a finite group $H$
in the sense of Definition \ref{df:action}.
Let $v_0$ be the vertex fixed by the action of $H$,
which exists by Assumption \ref{as:vertex},
and $\theta$ be a $v_0$-generated stability parameter.

%so that $H$ acts on $\cM_\theta$.

\begin{lemma}
There is an action of $H$ on $\cE$ which is compatible with the action $\nu$ on $\cM_\theta$.
Therefore $E$ is an $H$-equivariant sheaf on $\Spec R$.
\end{lemma}

\begin{proof}
As in \cite[\S 2.1]{MR2078369},
the moduli space $\cM_\theta$ is constructed as a quotient of the scheme 
\[
\scN_\theta \subset \prod_{a \in Q_1} \Hom_{\bC}(V_{s(a)}, V_{t(a)})
\]
parametrizing $\theta$-stable representations of $\Gamma$ in vector spaces $V_v=\bC$ for $v \in V$ by
the action of the group
\[
\Aut'((V_v)_{v \in Q_0})\coloneqq \left\{\left.(g_v)_{v \in Q_0} \in \prod_{v \in Q_0} \GL(V_v) \,\right|\, g_{v_0}=1 \right\}
\cong \prod_{v \in Q_0 \setminus \{v_0\}} \GL(V_v).
\]
This group $\Aut'((V_v)_{v \in Q_0})$ acts on the locally free sheaf
$\cEtilde\coloneqq\bigoplus_v V_v \otimes \cO_{\scN_\theta}$ on $\scN_\theta$
and $\cEtilde$ descends to the tautological bundle $\cE$ on $\cM_\theta$.
On the other hand,  we can define an action $\tilde{\nu}$ of $H$ on $\scN_\theta$ by changing the sign in the natural action
as in Section \ref{sc:action} which is compatible with the action $\nu$ on $\cM_\theta$.
We can also let $H$ act on the group $\Aut'((V_v)_{v \in Q_0})$ by
\[
(h, (g_v)_{v \in Q_0}) \mapsto (g_{(h^{-1}(v))})_{v \in Q_0}
\]
and on $\cEtilde$ by
\[
\left(h, \,\bigoplus_v w_v \otimes f_v\right) \mapsto \bigoplus_v w_{h^{-1}(v)} \otimes \tilde{\nu}(h, f_{h^{-1}(v)}).
\]
Thus the semidirect product
$H \ltimes (\Aut'((V_v)_{v \in Q_0}))$ acts on $\bigoplus_v V_v \otimes \cO_{\scN_\theta}$
which descends to an action of $H$  on $\cE$.
\end{proof}

Let us now give the proof of Theorem \ref{th:NCCR}. We first consider $H\ltimes\bC \Gamma$
by using the action of $H$ on $\bC \Gamma$.
In what follows we prove that $H\ltimes\bC \Gamma\cong H\ltimes\End_R(E)$ is a NCCR of $R^H$. 
According to \cite{MR2077594}, it is sufficient to show the following:
\begin{itemize}
\item $H\ltimes\bC\Gamma\cong\End_{R^H}(E)$,
\item $E$ is a reflexive $R^H$-module,
\item $H\ltimes\bC\Gamma$ is Cohen--Macaulay,
\item $H\ltimes\bC\Gamma$ has finite global dimension. 
\end{itemize}

Since $\bC\Gamma$ is Cohen--Macaulay over $R$, the crossed product $H\ltimes\bC\Gamma$ is also Cohen--Macaulay over $R$. 
Thus $H\ltimes\bC\Gamma$ is Cohen--Macaulay over $R^H$.
Similarly, $E$ is reflexive over $R$ and hence reflexive over $R^H$. 
Moreover, since $\bC\Gamma$ has finite global dimension, $H\ltimes\bC\Gamma$ has also finite global dimension. 
It is remaining to prove that $H\ltimes\bC\Gamma\cong\End_{R^H}(E)$.
% where $E$ is Cohen--Macaulay over $R^H$. 
%The last assertion is immediate since $E$ is Cohen--Macaulay over $R$, so it is remaining to prove the isomorphism. 
%We divide the proof into 4 steps.

Notice that $\bC\Gamma\cong\End_R(E)\subseteq\End_{R^H}(E)$ and the action of $H$ on $E$ induces a monoid homomorphism $H\to\End_{R^H}(E)$. Therefore, there exists an algebra homomorphism 
\[
\scF:H\ltimes\End_R(E)\to\End_{R^H}(E).
\]

Recall that both $H\ltimes\End_R(E)$ and $\End_{R^H}(E)$ are reflexive $R^H$-modules.
Therefore it suffices to prove that $\scF$ is an isomorphism over some open subset $U\subset\Spec R^H$
with $\codim (\Spec R^H\backslash U)\geq2$.
To choose this open set,
let $\widetilde{U}$ be the smooth and $H$-free locus in $\Spec R$ and define $U\coloneqq\widetilde{U}/H$. 
The isomorphism $K_{\cM_\theta} \cong \cO_{\cM_\theta}$ in $\coh^H(\cM_\theta)$
implies that $\codim (\Spec R^H\backslash U) \ge 2$.

%\begin{center}4
%\begin{pspicture}(0,0)(2,2.25)
%	\psset{nodesep=3pt}
%	\rput(0,2){$H\cdot\widetilde{P}\subset\widetilde{U}\subset\Spec R$}
%	\rput(0.45,0){$P\in U\subset\Spec R^H$}
%	\psline{->}(1.25,1.7)(1.25,0.3)
%	\psline{->}(-0.25,1.7)(-0.25,0.3)
%	\rput(1.5,1){$\pi$}
%\end{pspicture}
%\end{center}
%\[
%\begin{tikzcd}
% H \cdot \widetilde{P} \arrow[r,phantom,"\subset"] &[-8mm] \widetilde{U} \arrow[r,hookrightarrow] \arrow[d] & \Spec R \arrow[d,"\pi"]\\
% p \arrow[r,phantom,"\in"] & U \arrow[r,hookrightarrow] & \Spec R^H
%\end{tikzcd}
%\]

We show that for every point $P\in U$,
the fibre of $\scF$ over $P$ is an isomorphism. 
Since the restriction of $\tau$ to $\tau^{-1}(\Utilde)$ is an isomorphism, the sheaf $E|_{\widetilde{U}}$ is locally free.
Moreover, since $\pi^{-1}(P)$ is a free $H$-orbit, we have $\End_R(E)|_P \cong \bigoplus_{Q \in \pi^{-1}(P)}\End_\bC(E|_{Q})$.
%
%for any point $\widetilde{p}\in\widetilde{U}$ we have that $E|_{G\cdot\widetilde{p}}\cong\bigoplus_{v\in Q_0}\cL_v|_{G\cdot\widetilde{p}}$. Then
%
%\begin{align*}
%\dim\End_R(E)|_{G\cdot\widetilde{p}} &= \dim\End_{\bC[G\cdot\widetilde{p}]}(E|_{G\cdot\widetilde{p}}) = \dim\bigoplus_{g\in G}\End_\bC(E|_{g\cdot\widetilde{p}})=|G||Q_0|^2 \\
%\dim\End_{R^G}(E)|_p &= \dim\End_\bC((\pi_*E)|_p) = \dim\End_\bC(\bigoplus_{g\in G}E|_{g\cdot\widetilde{p}}) = |G|^2|Q_0|^2
%\end{align*}
%which means that $\dim(G\ast\End_{R}(E)|_{G\cdot\widetilde{p}}) = \dim\End_{R^G}(E)|_p$. Then to show that $G\ast\End_R(E|_{G\cdot\widetilde{p}})\cong\End_\bC(\pi_*E|_p)$, it is enough to prove that the following map is surjective
%
Thus the problem is reduced to showing that the map
\[
\scF|_{P}:H\ltimes\left(\bigoplus_{Q \in \pi^{-1}(P)}\End_\bC(E|_{Q})\right)
\to \End_\bC\left(\bigoplus_{Q \in \pi^{-1}(P)}E|_{Q}\right)
\]
is an isomorphism of vector spaces.
The left hand side, as a vector space, decomposes as
\[
H\ltimes\left(\bigoplus_{Q \in \pi^{-1}(P)}\End_\bC(E|_{Q})\right)
=\bigoplus_{g \in H, \, Q \in \pi^{-1}(P)} g \ltimes \End_\bC(E|_{Q}),
\]
and $\scF|_{P}$ sends the direct summand $g \ltimes \End_\bC(E|_{Q})$ isomorphically onto
$\Hom_\bC(E|_{Q}, E|_{gQ})$.
Since $\pi^{-1}(P)$ is a free $H$-orbit, $\scF|_P$ is an isomorphism.
This concludes the proof of Theorem \ref{th:NCCR}.

\begin{remark}\label{rem:fixed_point}
Assumption \ref{as:polygon} is used only when $H$ does not preserve the orientation of $T$.
In fact, when $H$ preserves the orientation of $T$, 
Theorem \ref{th:NCCR} holds under only Assumption \ref{as:vertex}.
\end{remark}

\begin{example}
Let $G$ be the $R_2$-symmetric consistent dimer model and $\Delta$ its $R_2$-symmetric characteristic polygon shown in Figure \ref{fg:pentagon_dimer}.
Let $X_{2k \Delta}$ be the affine toric variety
associated with the polygon
$
2 k \Delta \coloneqq \lc 2 k n \in N_\bR \relmid n \in \Delta \rc
$
obtained by multiplying $\Delta$ by $2 k$ for $k \ge 2$
(or
the same polygon $\Delta$
regarded as a lattice polygon
with respect to the over-lattice
$
\frac{1}{2 k} N \coloneqq \lc n \in N_\bR \relmid 2 k n \in N \rc
$
of index $(2k)^2$).
Let $P_0, \ldots, P_4$
be the corners of $2 k \Delta$
such that the reflection fixes $P_0$ and
interchanges $P_1$ (resp.~$P_2$) with $P_4$ (resp.~$P_3$).
The affine toric variety $X_{2k \Delta}$ has
a family of $A_{2k-1}$-singularities
along the torus-invariant curve
associated with 
the $R_2$-invariant face of $2 k \Delta$
(i.e., the edge connecting $P_2$ and $P_3$).
We choose the midpoint of $P_2$ and $P_3$ as the origin which defines the twisted action
of $R_2$ on $X_{2k \Delta}$ as in \eqref{eq:twist_on_X}.
Then the quotient of the family of $A_{2k-1}$-singularities
gives
a family of $D_{k+3}$-singularity along a curve
on $X_{2k \Delta}/R_2$ (see Remark \ref{rem:choice}).
The existence of a family of $D_{k+3}$-singularities along a curve
implies that the affine variety
$X_{2k \Delta}/R_2$
is not a toric variety
(if $X_{2k \Delta}/R_2$ is a toric variety,
then the curve along which the variety has $D_{k+3}$-singularities
must be torus-invariant,
so that the affine toric variety
associated with the 2-dimensional cone
corresponding to that curve
must have a $D_{k+3}$-singularity,
which is impossible).
To prove that (the origin of)
$X_{2k \Delta}/R_2$
is not a quotient singularity,
we show that
$X_{2k \Delta}/R_2$
is not $\bQ$-factorial.
To prove that a variety is not $\bQ$-factorial,
it suffices to find a pair of Weil divisors
intersecting in codimension at least three.
For each $0 \le i \le 4$,
let $D_i$ be the divisor of $X_{2 k \Delta}$
associated with $P_i$.
Then the divisors $D_0$ and $D_2 + D_3$ descends to
divisors on
$X_{2 k \Delta}/R_2$
intersecting in codimension three.
Hence
the symmetric dimer model
obtained from $G$ by Lemma \ref{lm:UY}
produces a non-commutative crepant resolution
of
$X_{2 k \Delta}/R_2$
which is neither a toric variety nor a quotient singularity.

\begin{figure}[h]
\[
\begin{array}{cc}
\begin{tikzpicture}[scale=0.75]
	\draw[fill=black] (0,0) circle  (1mm);
	\draw[fill=black] (1,0) circle   (1mm);
	\draw[fill=black] (0,1) circle   (1mm);
	\draw[fill=black] (1,1) circle   (1mm);
	\draw[fill=black] (1,2) circle   (1mm);
	\draw[fill=black] (2,1) circle   (1mm);	

	\node at (-0.25,-0.4) {$P_0$};
	\node at (1.25,-0.4) {$P_1$};
	\node at (-0.5,1) {$P_4$};
	\node at (1,2.4) {$P_3$};
	\node at (2.5,1) {$P_2$};	

	\draw[thick,-] (0,0) -- (1,0) -- (2,1) -- (1,2) -- (0,1) -- cycle;
	\node at (0,-0.5) {};	
\end{tikzpicture}
~~~~~~~~~~~~~~~~&
\begin{tikzpicture}[scale=1.2,bend angle=45, looseness=1]

	\def\h{0.9238} %height of the octagon
	\def\l{0.76537} %length of side of the octagon = side of square
	\def\c{0.5412} %length of side of the octagon = side of square

 	\draw[thick] (3/2*\h,-\l/2) -- (\h,-\l/2) -- (-\l/2,\h) -- (-\l/2,3/2*\h);
 	\draw[thick] (\l/2,3/2*\h) -- (\l/2,\h) -- (\h,\l/2) -- (3/2*\h,\l/2);
	\draw[thick] (-3/2*\h,\l/2) -- (-\h,\l/2) -- (-\h,-\l/2) -- (-3/2*\h,-\l/2);
	\draw[thick] (-\l/2,-3/2*\h) -- (-\l/2,-\h) -- (\l/2,-\h) -- (\l/2,-3/2*\h);
		
		\octopo{((0,0)}{(0,0)}
		
	\draw[color=gray] (-3/2*\h,-3/2*\h) -- (3/2*\h,-3/2*\h) -- (3/2*\h,3/2*\h) --  (-3/2*\h,3/2*\h) -- cycle;		
%	\node at (\h+\l/2,0) {\gray 3};
%	\node at (0,-\h-\l/2) {\gray 2};
%	\node at (\h+\l/2,-\h-\l/2) {\gray 0};
\end{tikzpicture}
\end{array}
\]
\caption{An $R_2$-symmetric dimer model with a pentagon as characteristic polygon.}
\label{fg:pentagon_dimer}
\end{figure}

\end{example}

\section{Wallpaper groups}
 \label{sc:wallpaper}

If a dimer model $G$
on the real 2-torus $M_\bR/M$
is symmetric
under the action of a finite subgroup $H$ of $\GL(M)\ltimes (M_\bR/M)$,
then we can think of the quotient graph $G/H$
on the 2-dimensional orbifold
$M_\bR/(H \ltimes M)$.
If $H$ contains a reflection or a glide reflection,
then the graph $G/H$ is no longer bicolored
and hence not a dimer model,
but the associated quiver with relation still makes sense
and can be drawn
on the orbifold $(M_\bR/M)/H$.
Dimer models and quivers on orbifolds are also discussed
by Bocklandt \cite{MR3010162}
under the name \emph{weighted quiver polyhedra},
which are different from the ones appearing in this paper
in that we allow reflections whereas he does not,
and that we allow orbifold points to lie on dimer edges and dimer faces
(i.e., quiver arrows and quiver vertices),
whereas orbifold points in his theory lie only on dimer nodes
(i.e., quiver faces).

A discrete subgroup $\mathsf{W}$ of the Euclidean group
$\operatorname{E}(2) = \operatorname{O}(2) \ltimes \bR^2$
containing two linearly independent translations
is called
a \emph{wallpaper group}
or a \emph{plane crystallographic group}.
Wallpaper groups are classified into 17 classes
by the diffeomorphism class of the orbifold quotient $\bR^2/\mathsf{W}$,
and described by the \emph{orbifold notation}
as in Table \ref{tb:orbifold_notation}
(cf.~e.g.~\cite{Conway}).

\begin{table}[h]
\centering
\begin{tabular}{cc}
\toprule
${\blue 0}$ & Translation of the fundamental domain \\
${\red \times}$ & Line of glide reflection \\
${\red \ast}$ & line of reflection symmetry\\
${\red n}$ after ${\red \ast}$ & Point passing $n$ lines of reflection symmetries \\
${\blue n}$ before ${\red \ast}$ & Center point of an order $n$ rotation symmetry \\
\bottomrule
\end{tabular}
\caption{The orbifold notation}
\label{tb:orbifold_notation}
\end{table}

When a dimer model on $M_\bR/M$ is symmetric
under the action of a finite subgroup $H$ of $\GL(M)\ltimes (M_\bR/M)$, %$\GL(N)$,
we can take a $H$-invariant metric on $M_\bR$,
so that the pull-back of the dimer model to the universal cover $M_\bR$
is invariant under a wallpaper group.
Conversely,
for each of 17 wallpaper groups,
one can ask if there is a \emph{consistent} dimer model
whose group of symmetries is given by that group.
The answer to this question is affirmative,
and we give an example of a consistent dimer model of each type
in Figure \ref{fg:dimer17} below.
Note that
the type of symmetry of a dimer model depends
not only on the isomorphism class of the underlying abstract graph,
or even the isotopy class of the embedding of the graph on the 2-torus,
but also on the isometry class of the embedding.
For example, in Figure \ref{fg:dimer17}, we see that ${\blue 442}$, ${\red\ast442}$, and ${\blue4}{\red\ast2}$ are isotopic, but have different symmetries. 
%Notice also that among the 17 dimer models shown in the figure the cases $\ast\ast$, $\ast2222$, $\ast333$, $\ast422$ and $\ast632$ do correspond to groups of symmetries in $\GL(2,\bZ)$. The rest of the cases either contain translations or glides, or have rotational symmetries with center not a vertex of the dimer model. 

\begin{figure}[h]
\[
\begin{array}{ccccc}

%%%%%%%%%%%%%%%%%%%%%%%%%
%       0
%%%%%%%%%%%%%%%%%%%%%%%%%
\begin{tikzpicture}[scale=0.45,bend angle=45, looseness=1]
	\draw (60:1) -- (\AA+60:\RootOfSeven);
	\draw (120:2) -- (180:2);
	\draw (0:1) -- (\AA:\RootOfSeven);
	\draw (240:1) -- (\AA+240:\RootOfSeven);
	% Center hexagon
	\hex{((0,0)}{(0,0)}
	% First ring of hexagons
	\foreach \a in {30,90,...,330} {\hex{(0,0)}{(\a:\RootOfThree)}}
\end{tikzpicture}

& 
%%%%%%%%%%%%%%%%%%%%%%%%%
%      ast ast
%%%%%%%%%%%%%%%%%%%%%%%%%

\begin{tikzpicture}[scale=0.5,bend angle=45, looseness=1]

	\def\h{0.9238} %height of the octagon
	\def\l{0.76537} %length of side of the octagon = side of square
	\def\c{0.5412} %length of side of the octagon = side of square

 	\draw  (\h,-\l/2) -- (-\l/2,\h);\draw  (3*\h+\l,-\l/2) -- (\h+\l+\c,\h);

		\oct{((0,0)}{(0,0)}
		\octop{((0,0)}{(22.5:1.8478)}
		\octop{((0,0)}{(-3*22.5:1.8478)}
		\oct{((0,0)}{(-22.5:1.8478+0.76539)}
		
		\Square{(0,0)}{(-45+26.565:2.9214)}
		\Square{(0,0)}{(-45-26.565:2.9214)}
		\Square{(0,0)}{(45:1.3065)}
		\Square{(0,0)}{(-135:1.3065)}		
\end{tikzpicture}

& 
%%%%%%%%%%%%%%%%%%%%%%%%%
%      xx
%%%%%%%%%%%%%%%%%%%%%%%%%
\begin{tikzpicture}[scale=0.45,bend angle=45, looseness=1]
	\draw (-\AA:\RootOfSeven) -- (0:1) -- (\AA:\RootOfSeven);
	\draw (240:2) -- (180:2) -- (120:2);
	% Center hexagon
	\hex{((0,0)}{(0,0)}
	% First ring of hexagons
	\foreach \a in {30,90,...,330} {\hex{(0,0)}{(\a:\RootOfThree)}}
%	\node at (0,-3.5) {$\times\times$};
\end{tikzpicture}

&
%%%%%%%%%%%%%%%%%%%%%%%%%
%      ast x
%%%%%%%%%%%%%%%%%%%%%%%%%
\begin{tikzpicture}[scale=0.5,bend angle=45, looseness=1]

	\def\h{0.9238} %height of the octagon
	\def\l{0.76537} %length of side of the octagon = side of square
	\def\c{0.5412} %length of side of the octagon = side of square

 %	\draw  (-3,-3) -- (3,-3) -- (3,3) -- (-3,3) -- cycle;
	\draw (-2.25,-1.25) -- (-2.25,1.75) -- (-0.75,1.75) -- (-0.75,-2) -- (0.75,-2) -- (0.75,1.75) -- (2.25,1.75) -- (2.25,-1.25) -- (0.75,-1.25); \draw (0.75,-0.5) -- (-0.75,-0.5); \draw (-0.75,-1.25) -- (-2.25,-1.25); \draw (-2.25,0.25) -- (-0.75,0.25); \draw (-0.75,1) -- (0.75,1); \draw (0.75,0.25) -- (2.25,0.25);
	
	\draw[fill=black] (-2.25,-1.25) circle (1.2mm);\draw[fill=black] (-2.25,0.25) circle (1.2mm);
	\draw[fill=black] (-2.25,1.75) circle (1.2mm);\draw[fill=black] (-0.75,1) circle (1.2mm);	\draw[fill=black] (-0.75,-0.5) circle (1.2mm);\draw[fill=black] (-0.75,-2) circle (1.2mm);	\draw[fill=black] (0.75,1.75) circle (1.2mm);\draw[fill=black] (0.75,0.25) circle (1.2mm);	\draw[fill=black] (0.75,-1.25) circle (1.2mm);\draw[fill=black] (2.25,1) circle (1.2mm);	\draw[fill=black] (2.25,-0.5) circle (1.2mm);

	\draw[fill=white] (-2.25,1) circle (1.2mm);\draw[fill=white] (-2.25,-0.5) circle (1.2mm);	\draw[fill=white] (-0.75,1.75) circle (1.2mm);\draw[fill=white] (-0.75,0.25) circle (1.2mm);	\draw[fill=white] (-0.75,-1.25) circle (1.2mm);\draw[fill=white] (0.75,1) circle (1.2mm);	\draw[fill=white] (0.75,-0.5) circle (1.2mm);\draw[fill=white] (0.75,-2) circle (1.2mm);	\draw[fill=white] (2.25,1.75) circle (1.2mm);\draw[fill=white] (2.25,0.25) circle (1.2mm);	\draw[fill=white] (2.25,-1.25) circle (1.2mm);

\end{tikzpicture}

& 

%%%%%%%%%%%%%%%%%%%%%%%%%
%      2222
%%%%%%%%%%%%%%%%%%%%%%%%%
\begin{tikzpicture}[scale=0.45,bend angle=45, looseness=1]

	\def\h{0.9238} %height of the octagon
	\def\l{0.76537} %length of side of the octagon = side of square
	\def\c{0.5412} %length of side of the octagon = side of square

% 	Horizontal lines
	\draw (0,0) -- (4.5,0) -- cycle; 
	\draw (0,1.5) -- (4.5,1.5) -- cycle; 	
	\draw (0.5,2.25) -- (5,2.25) -- cycle; 
	\draw (0.5,3.75) -- (5,3.75) -- cycle; 	
	\draw (1,4.5) -- (5.5,4.5) -- cycle; 
% 	Vertical lines
	\draw (0,0) -- (0,1.5) -- (0.5,2.25) -- (0.5,3.75) -- (1,4.5); 
	\draw (1.5,0) -- (1.5,1.5) -- (2,2.25) -- (2,3.75) -- (2.5,4.5); 
	\draw (2.25,0) -- (2.25,1.5) -- (2.75,2.25) -- (2.75,3.75) -- (3.25,4.5); 
	\draw (3.75,0) -- (3.75,1.5) -- (4.25,2.25) -- (4.25,3.75) -- (4.75,4.5); 
	\draw (4.5,0) -- (4.5,1.5) -- (5,2.25) -- (5,3.75) -- (5.5,4.5); 
%	black nodes
	\draw[fill=white] (0,0) circle (1.2mm);	\draw[fill=black] (0,1.5) circle (1.2mm);	\draw[fill=white] (0.5,2.25) circle (1.2mm);	\draw[fill=black] (0.5,3.75) circle (1.2mm);	\draw[fill=white] (1,4.5) circle (1.2mm);	\draw[fill=black] (1.5,0) circle (1.2mm);	\draw[fill=white] (1.5,1.5) circle (1.2mm);	\draw[fill=black] (2,2.25) circle (1.2mm);	\draw[fill=white] (2,3.75) circle (1.2mm);	\draw[fill=black] (2.5,4.5) circle (1.2mm); 
	\draw[fill=white] (2.25,0) circle (1.2mm);	\draw[fill=black] (2.25,1.5) circle (1.2mm);	\draw[fill=white] (2.75,2.25) circle (1.2mm);	\draw[fill=black] (2.75,3.75) circle (1.2mm);	\draw[fill=white] (3.25,4.5) circle (1.2mm); 	\draw[fill=black] (3.75,0) circle (1.2mm);	\draw[fill=white] (3.75,1.5) circle (1.2mm);	\draw[fill=black] (4.25,2.25) circle (1.2mm);	\draw[fill=white] (4.25,3.75) circle (1.2mm);	\draw[fill=black] (4.75,4.5) circle (1.2mm);	\draw[fill=white] (4.5,0) circle (1.2mm);	\draw[fill=black] (4.5,1.5) circle (1.2mm);	\draw[fill=white] (5,2.25) circle (1.2mm);	\draw[fill=black] (5,3.75) circle (1.2mm);	\draw[fill=white] (5.5,4.5) circle (1.2mm);

\end{tikzpicture}

\\
{\blue 0} & {\red \ast\ast} & {\red \times\!\times} & {\red \ast\times} & {\blue 2222} 

\\ [0.25cm]

%%%%%%%%%%%%%%%%%%%%%%%%%
%      ast 2222
%%%%%%%%%%%%%%%%%%%%%%%%%
\begin{tikzpicture}[scale=0.5,bend angle=45, looseness=1]

	\def\h{0.9238} %height of the octagon
	\def\l{0.76537} %length of side of the octagon = side of square
	\def\c{0.5412} %length of side of the octagon = side of square

	\draw (-2,-1.5) -- (-2,1.5) -- (2,1.5) -- (2,-1.5) -- cycle;\draw (-2,-0.5) -- (2,-0.5) -- (2,0.5) -- (-2,0.5);\draw (0,-1.5) -- (0,1.5);
	
	\draw[fill=black] (-2,-0.5) circle (1.2mm);\draw[fill=black] (-2,1.5) circle (1.2mm);\draw[fill=black] (0,0.5) circle (1.2mm);
	\draw[fill=black] (0,-1.5) circle (1.2mm);\draw[fill=black] (2,1.5) circle (1.2mm);\draw[fill=black] (2,-0.5) circle (1.2mm);

	\draw[fill=white] (-2,-1.5) circle (1.2mm);\draw[fill=white] (-2,0.5) circle (1.2mm);\draw[fill=white] (0,1.5) circle (1.2mm);
	\draw[fill=white] (0,-0.5) circle (1.2mm);\draw[fill=white] (2,0.5) circle (1.2mm);\draw[fill=white] (2,-1.5) circle (1.2mm);
\end{tikzpicture}

&
%%%%%%%%%%%%%%%%%%%%%%%%%
%      22x
%%%%%%%%%%%%%%%%%%%%%%%%%
\begin{tikzpicture}[scale=0.5,bend angle=45, looseness=1]

	\def\h{0.9238} %height of the octagon
	\def\l{0.76537} %length of side of the octagon = side of square
	\def\c{0.5412} %length of side of the octagon = side of square

 %	\draw  (-3,-3) -- (3,-3) -- (3,3) -- (-3,3) -- cycle;
	\draw (-2.25,-1.5) -- (-2.25,1.5) -- (-0.75,2.25) -- (-0.75,-1.5) -- (0.75,-2.25) -- (0.75,1.5) -- (2.25,2.25) -- (2.25,-0.75) -- (0.75,-1.5) ;	\draw (-2.25,-1.5) -- (-0.75,-0.75);	\draw (-0.75,0) -- (0.75,-0.75);
	\draw (-2.25,0) -- (-0.75,0.75);\draw (-0.75,1.5) -- (0.75,0.75);\draw (0.75,0) -- (2.25,0.75);		
	
	\draw[fill=black] (-2.25,-1.5) circle (1.2mm);\draw[fill=black] (-0.75,-1.5) circle (1.2mm);\draw[fill=black] (0.75,-1.5) circle (1.2mm);	\draw[fill=black] (-2.25,0) circle (1.2mm);\draw[fill=black] (-0.75,0) circle (1.2mm);\draw[fill=black] (0.75,0) circle (1.2mm);	\draw[fill=black] (2.25,0) circle (1.2mm);\draw[fill=black] (-2.25,1.5) circle (1.2mm);draw[fill=black] (-0.75,1.5) circle (1.2mm);	\draw[fill=black] (0.75,1.5) circle (1.2mm);\draw[fill=black] (2.25,1.5) circle (1.2mm);

	\draw[fill=white] (-2.25,-0.75) circle (1.2mm);\draw[fill=white] (-2.25,0.75) circle (1.2mm);\draw[fill=white] (-0.75,-0.75) circle (1.2mm);	\draw[fill=white] (-0.75,0.75) circle (1.2mm);\draw[fill=white] (-0.75,2.25) circle (1.2mm);\draw[fill=white] (0.75,-0.75) circle (1.2mm);	\draw[fill=white] (0.75,0.75) circle (1.2mm);\draw[fill=white] (0.75,-2.25) circle (1.2mm);\draw[fill=white] (2.25,-0.75) circle (1.2mm);	\draw[fill=white] (2.25,0.75) circle (1.2mm);\draw[fill=white] (2.25,2.25) circle (1.2mm);
\end{tikzpicture}

&
%%%%%%%%%%%%%%%%%%%%%%%%%
%      22 ast
%%%%%%%%%%%%%%%%%%%%%%%%%
\begin{tikzpicture}[scale=0.5,bend angle=45, looseness=1]

	\def\h{0.9238} %height of the octagon
	\def\l{0.76537} %length of side of the octagon = side of square
	\def\c{0.5412} %length of side of the octagon = side of square

% 	\draw  (-3,-3) -- (3,-3) -- (3,3) -- (-3,3) -- cycle;
	\draw (-2,-1.5) -- (-2,-0.5) -- (2,-0.5) -- (2,0.5) -- (-2,0.5) -- (-2,1.5) -- (1.5,1.5) -- (1.5,0.5);
	\draw (-2,-1.5) -- (1.5,-1.5) -- (1.5,-0.5); \draw (0.5,-0.5) -- (0.5,0.5); \draw (0,-1.5) -- (0,1.5); \draw (-0.5,-1.5) -- (-0.5,-0.5); \draw (-0.5,0.5) -- (-0.5,1.5); \draw (-1.5,-0.5) -- (-1.5,0.5);
		
	\draw[fill=black] (-2,-0.5) circle (1.2mm);\draw[fill=black] (-2,1.5) circle (1.2mm);\draw[fill=black] (-1.5,0.5) circle (1.2mm);
	\draw[fill=black] (-1.5,-1.5) circle (1.2mm);\draw[fill=black] (-0.5,1.5) circle (1.2mm);\draw[fill=black] (-0.5,-0.5) circle (1.2mm);
	\draw[fill=black] (0,0.5) circle (1.2mm);\draw[fill=black] (0,-1.5) circle (1.2mm);\draw[fill=black] (0.5,1.5) circle (1.2mm);
	\draw[fill=black] (0.5,-0.5) circle (1.2mm);\draw[fill=black] (1.5,0.5) circle (1.2mm);\draw[fill=black] (1.5,-1.5) circle (1.2mm);	
	\draw[fill=black] (2,-0.5) circle (1.2mm);	

	\draw[fill=white] (-2,0.5) circle (1.2mm);\draw[fill=white] (-2,-1.5) circle (1.2mm);\draw[fill=white] (-1.5,-0.5) circle (1.2mm);
	\draw[fill=white] (-1.5,1.5) circle (1.2mm);\draw[fill=white] (-0.5,-1.5) circle (1.2mm);\draw[fill=white] (-0.5,0.5) circle (1.2mm);
	\draw[fill=white] (0,1.5) circle (1.2mm);\draw[fill=white] (0,-0.5) circle (1.2mm);\draw[fill=white] (0.5,-1.5) circle (1.2mm);
	\draw[fill=white] (0.5,0.5) circle (1.2mm);\draw[fill=white] (1.5,-0.5) circle (1.2mm);\draw[fill=white] (1.5,1.5) circle (1.2mm);	
	\draw[fill=white] (2,0.5) circle (1.2mm);	
\end{tikzpicture}

& 

%%%%%%%%%%%%%%%%%%%%%%%%%
%      2 ast 22
%%%%%%%%%%%%%%%%%%%%%%%%%
\begin{tikzpicture}[scale=0.45,bend angle=45, looseness=1]

	\draw (-0.25,0.75) -- (0.25,0.75) -- (0.25,2.25) -- (-0.25,2.25) -- cycle;
	\draw (-0.25,-0.75) -- (0.25,-0.75) -- (0.25,-2.25) -- (-0.25,-2.25) -- cycle;
	\draw (-1.75,-0.75) -- (-1.25,-0.75) -- (-1.25,0.75) -- (-1.75,0.75) -- cycle;
	\draw (1.75,-0.75) -- (1.25,-0.75) -- (1.25,0.75) -- (1.75,0.75) -- cycle;
	\draw (-2.75,0.75) -- (-2.75,2.25); \draw (-2.75,-0.75) -- (-2.75,-2.25);
	\draw (2.75,0.75) -- (2.75,2.25); \draw (2.75,-0.75) -- (2.75,-2.25);
	\draw (-1.75,2.75) -- (-1.75,2.25) -- (-1.25,2.25) -- (-1.25,2.75);
	\draw (1.75,2.75) -- (1.75,2.25) -- (1.25,2.25) -- (1.25,2.75);
	\draw (-1.75,-2.75) -- (-1.75,-2.25) -- (-1.25,-2.25) -- (-1.25,-2.75);
	\draw (1.75,-2.75) -- (1.75,-2.25) -- (1.25,-2.25) -- (1.25,-2.75);
	\draw (-2.75,2.25) -- (-1.75,2.75); \draw (2.75, 2.25) -- (1.75, 2.75);\draw (-2.75,-2.25) -- (-1.75,-2.75); \draw (2.75,-2.25) -- (1.75,-2.75);
	\draw (-2.75,1.75) -- (-1.75,2.25);\draw (-2.75,-1.75) -- (-1.75,-2.25);\draw (2.75,-1.75) -- (1.75,-2.25);\draw (2.75,1.75) -- (1.75,2.25);
	\draw (-1.25,2.75) -- (-0.25,2.25);\draw (1.25,2.75) -- (0.25,2.25);\draw (1.25,-2.75) -- (0.25,-2.25);\draw (-1.25,-2.75) -- (-0.25,-2.25);
	\draw (0.25,1.75) -- (1.25,2.25);\draw (0.25,-1.75) -- (1.25,-2.25);\draw (-0.25,1.75) -- (-1.25,2.25);\draw (-0.25,-1.75) -- (-1.25,-2.25);
	\draw (0.25,0.75) -- (1.25,0.25);\draw (0.25,-0.75) -- (1.25,-0.25);\draw (-0.25,0.75) -- (-1.25,0.25);\draw (-0.25,-0.75) -- (-1.25,-0.25);
	\draw (0.25,1.25) -- (1.25,0.75);\draw (-0.25,1.25) -- (-1.25,0.75);\draw (0.25,-1.25) -- (1.25,-0.75);\draw (-0.25,-1.25) -- (-1.25,-0.75);
	\draw (1.75,0.75) -- (2.75,1.25);\draw (-1.75,0.75) -- (-2.75,1.25);\draw (1.75,-0.75) -- (2.75,-1.25);\draw (-1.75,-0.75) -- (-2.75,-1.25);
	\draw (1.75,0.25) -- (2.75,0.75);\draw (-1.75,0.25) -- (-2.75,0.75);\draw (1.75,-0.25) -- (2.75,-0.75);\draw (-1.75,-0.25) -- (-2.75,-0.75);

	\draw[fill=black] (-2.75,1.75) circle (1.2mm);\draw[fill=black] (-2.75,0.75) circle (1.2mm);	\draw[fill=black] (-2.75,-1.25) circle (1.2mm);\draw[fill=black] (-2.75,-2.25) circle (1.2mm);	\draw[fill=white] (-2.75,2.25) circle (1.2mm);\draw[fill=white] (-2.75,1.25) circle (1.2mm);	\draw[fill=white] (-2.75,-0.75) circle (1.2mm);\draw[fill=white] (-2.75,-1.75) circle (1.2mm);	\draw[fill=black] (-1.75,2.75) circle (1.2mm);\draw[fill=black] (-1.75,0.75) circle (1.2mm);	\draw[fill=black] (-1.75,-0.25) circle (1.2mm);\draw[fill=black] (-1.75,-2.25) circle (1.2mm);	\draw[fill=white] (-1.75,2.25) circle (1.2mm);\draw[fill=white] (-1.75,0.25) circle (1.2mm);	\draw[fill=white] (-1.75,-0.75) circle (1.2mm);\draw[fill=white] (-1.75,-2.75) circle (1.2mm);	\draw[fill=white] (-1.25,2.75) circle (1.2mm);\draw[fill=white] (-1.25,0.75) circle (1.2mm);	\draw[fill=white] (-1.25,-0.25) circle (1.2mm);\draw[fill=white] (-1.25,-2.25) circle (1.2mm);	\draw[fill=black] (-1.25,2.25) circle (1.2mm);\draw[fill=black] (-1.25,0.25) circle (1.2mm);	\draw[fill=black] (-1.25,-0.75) circle (1.2mm);\draw[fill=black] (-1.25,-2.75) circle (1.2mm);	\draw[fill=black] (-0.25,2.25) circle (1.2mm);\draw[fill=black] (-0.25,1.25) circle (1.2mm);	\draw[fill=black] (-0.25,-0.75) circle (1.2mm);\draw[fill=black] (-0.25,-1.75) circle (1.2mm);	\draw[fill=white] (-0.25,1.75) circle (1.2mm);\draw[fill=white] (-0.25,0.75) circle (1.2mm);	\draw[fill=white] (-0.25,-1.25) circle (1.2mm);\draw[fill=white] (-0.25,-2.25) circle (1.2mm);	\draw[fill=black] (0.25,1.75) circle (1.2mm);\draw[fill=black] (0.25,0.75) circle (1.2mm);	\draw[fill=black] (0.25,-1.25) circle (1.2mm);\draw[fill=black] (0.25,-2.25) circle (1.2mm);	\draw[fill=white] (0.25,2.25) circle (1.2mm);\draw[fill=white] (0.25,1.25) circle (1.2mm);	\draw[fill=white] (0.25,-0.75) circle (1.2mm);\draw[fill=white] (0.25,-1.75) circle (1.2mm);	\draw[fill=black] (1.25,2.75) circle (1.2mm);\draw[fill=black] (1.25,0.75) circle (1.2mm);	\draw[fill=black] (1.25,-0.25) circle (1.2mm);\draw[fill=black] (1.25,-2.25) circle (1.2mm);	\draw[fill=white] (1.25,2.25) circle (1.2mm);\draw[fill=white] (1.25,0.25) circle (1.2mm);	\draw[fill=white] (1.25,-0.75) circle (1.2mm);\draw[fill=white] (1.25,-2.75) circle (1.2mm);	\draw[fill=white] (1.75,2.75) circle (1.2mm);\draw[fill=white] (1.75,0.75) circle (1.2mm);	\draw[fill=white] (1.75,-0.25) circle (1.2mm);\draw[fill=white] (1.75,-2.25) circle (1.2mm);	\draw[fill=black] (1.75,2.25) circle (1.2mm);\draw[fill=black] (1.75,0.25) circle (1.2mm);	\draw[fill=black] (1.75,-0.75) circle (1.2mm);\draw[fill=black] (1.75,-2.75) circle (1.2mm);	\draw[fill=black] (2.75,2.25) circle (1.2mm);\draw[fill=black] (2.75,1.25) circle (1.2mm);	\draw[fill=black] (2.75,-0.75) circle (1.2mm);\draw[fill=black] (2.75,-1.75) circle (1.2mm);	\draw[fill=white] (2.75,1.75) circle (1.2mm);\draw[fill=white] (2.75,0.75) circle (1.2mm);	\draw[fill=white] (2.75,-1.25) circle (1.2mm);\draw[fill=white] (2.75,-2.25) circle (1.2mm);	

\end{tikzpicture}

& 
%%%%%%%%%%%%%%%%%%%%%%%%%
%       333
%%%%%%%%%%%%%%%%%%%%%%%%%
\begin{tikzpicture}[scale=0.275,bend angle=45, looseness=1]

	\draw (-0.5,-\hh) -- (1,-2*\hh) -- (2,-2*\hh) -- (2.5,-\hh) -- (2,0) -- (1,0) -- (1,2*\hh) -- (0.5,3*\hh) -- (-0.5,3*\hh) -- (-1,2*\hh) -- (-0.5,\hh) -- (-2,0) -- (-2.5,-\hh) -- (-2,-2*\hh) -- (-1,-2*\hh) -- cycle;
	\draw (-4,0) -- (-5,0) -- (-5.5,\hh) -- (-5,2*\hh) -- (-3.5,\hh) -- (-2.5,\hh) -- (-2,2*\hh);
	\draw (-2.5,3*\hh) -- (-3.5,3*\hh) -- (-3.5,5*\hh) -- (-2.5,5*\hh) -- (-2,4*\hh);
	\draw (-1,4*\hh) -- (-0.5,5*\hh) -- (1,4*\hh);
	\draw (2,4*\hh) -- (2.5,5*\hh) -- (3.5,5*\hh) -- (4,4*\hh) -- (2.5,3*\hh) -- (2,2*\hh) -- (2.5,\hh);
	\draw (3.5,\hh) -- (4,2*\hh) -- (5.5,\hh) -- (5,0) -- (4,0);
	\draw (2.5,-3*\hh) -- (4,-2*\hh) -- (3.5,-\hh);
	\draw (2,-4*\hh) -- (2.5,-5*\hh) -- (2,-6*\hh) -- (1,-6*\hh) -- (1,-4*\hh) -- (0.5,-3*\hh) -- (-0.5,-3*\hh);
	\draw (-1,-4*\hh) -- (-0.5,-5*\hh) -- (-2,-6*\hh) -- (-2.5,-5*\hh) -- (-2,-4*\hh);
	\draw (-2.5,-3*\hh) -- (-3.5,-3*\hh) -- (-3.5,-\hh);
	
	\draw (2,0) -- (2.5,\hh) -- (3.5,\hh) -- (4,0) -- (3.5,-\hh) -- (2.5,-\hh) -- cycle;
	\draw (1,2*\hh) -- (0.5,3*\hh) -- (1,4*\hh) -- (2,4*\hh) -- (2.5,3*\hh) -- (2,2*\hh) -- cycle;
	\draw (-2,2*\hh) -- (-2.5,3*\hh) -- (-2,4*\hh) -- (-1,4*\hh) -- (-0.5,3*\hh) -- (-1,2*\hh) -- cycle;
	\draw (-2.5,-\hh) -- (-3.5,-\hh) -- (-4,0) -- (-3.5,\hh) -- (-2.5,\hh) -- (-2,0) -- cycle;
	\draw (-1,-2*\hh) -- (-0.5,-3*\hh) -- (-1,-4*\hh) -- (-2,-4*\hh) -- (-2.5,-3*\hh) -- (-2,-2*\hh) -- cycle;
	\draw (1,-2*\hh) -- (2,-2*\hh) -- (2.5,-3*\hh) -- (2,-4*\hh) -- (1,-4*\hh) -- (0.5,-3*\hh) -- cycle;
%	\foreach \a in {0,60,120,...,300} {\hex{(0,0)}{(\a:3)}}

	\draw[fill=white] (-0.5,-\hh) circle (1.6mm);
	\draw[fill=white] (1,0) circle (1.6mm);
	\draw[fill=white] (-0.5,\hh) circle (1.6mm);
	\draw[fill=white] (2.5,5*\hh) circle (1.6mm);
	\draw[fill=white] (4,4*\hh) circle (1.6mm);
	\draw[fill=white] (4,2*\hh) circle (1.6mm);
	\draw[fill=white] (5.5,\hh) circle (1.6mm);	
	\draw[fill=white] (4,-2*\hh) circle (1.6mm);		
	\draw[fill=white] (2.5,-5*\hh) circle (1.6mm);
	\draw[fill=white] (1,-6*\hh) circle (1.6mm);	
	\draw[fill=white] (-0.5,-5*\hh) circle (1.6mm);	
	\draw[fill=white] (-2,-6*\hh) circle (1.6mm);	
	\draw[fill=white] (-3.5,-3*\hh) circle (1.6mm);	
	\draw[fill=white] (-5,0) circle (1.6mm);	
	\draw[fill=white] (-5,2*\hh) circle (1.6mm);	
	\draw[fill=white] (-3.5,5*\hh) circle (1.6mm);	
	\draw[fill=white] (-3.5,3*\hh) circle (1.6mm);	
	\draw[fill=white] (-0.5,5*\hh) circle (1.6mm);	
	\draw[fill=white] (1,4*\hh) circle (1.6mm);	
	\draw[fill=white] (2.5,3*\hh) circle (1.6mm);	
	\draw[fill=white] (1,2*\hh) circle (1.6mm);	
	\draw[fill=white] (-2,4*\hh) circle (1.6mm);	
	\draw[fill=white] (-0.5,3*\hh) circle (1.6mm);	
	\draw[fill=white] (-2,2*\hh) circle (1.6mm);	
	\draw[fill=white] (-3.5,\hh) circle (1.6mm);	
	\draw[fill=white] (-2,0) circle (1.6mm);	
	\draw[fill=white] (-3.5,-\hh) circle (1.6mm);	
	\draw[fill=white] (-2,-2*\hh) circle (1.6mm);	
	\draw[fill=white] (-0.5,-3*\hh) circle (1.6mm);	
	\draw[fill=white] (-2,-4*\hh) circle (1.6mm);	
	\draw[fill=white] (1,-2*\hh) circle (1.6mm);	
	\draw[fill=white] (2.5,-3*\hh) circle (1.6mm);	
	\draw[fill=white] (1,-4*\hh) circle (1.6mm);	
	\draw[fill=white] (2.5,\hh) circle (1.6mm);	
	\draw[fill=white] (4,0) circle (1.6mm);	
	\draw[fill=white] (2.5,-\hh) circle (1.6mm);

	\draw[fill=black] (-3.5,4*\hh) circle (1.6mm);
	\draw[fill=black] (-2.5,5*\hh) circle (1.6mm);
	\draw[fill=black] (0.25,4.5*\hh) circle (1.6mm);
	\draw[fill=black] (3.5,5*\hh) circle (1.6mm);
	\draw[fill=black] (3.25,3.5*\hh) circle (1.6mm);
	\draw[fill=black] (4.75,1.5*\hh) circle (1.6mm);
	\draw[fill=black] (1,\hh) circle (1.6mm);
	\draw[fill=black] (3.25,-2.5*\hh) circle (1.6mm);
	\draw[fill=black] (2,-6*\hh) circle (1.6mm);
	\draw[fill=black] (1,-5*\hh) circle (1.6mm);
	\draw[fill=black] (0.25,-1.5*\hh) circle (1.6mm);
	\draw[fill=black] (-1.25,-5.5*\hh) circle (1.6mm);
	\draw[fill=black] (-2.5,-5*\hh) circle (1.6mm);
	\draw[fill=black] (-3.5,-2*\hh) circle (1.6mm);
	\draw[fill=black] (-5.5,\hh) circle (1.6mm);
	\draw[fill=black] (-4.25,1.5*\hh) circle (1.6mm);
	\draw[fill=black] (-1.25,0.5*\hh) circle (1.6mm);
	\draw[fill=black] (5,0) circle (1.6mm);
	\draw[fill=black] (2,4*\hh) circle (1.6mm);
	\draw[fill=black] (2,2*\hh) circle (1.6mm);
	\draw[fill=black] (0.5,3*\hh) circle (1.6mm);
	\draw[fill=black] (-1,4*\hh) circle (1.6mm);
	\draw[fill=black] (-1,2*\hh) circle (1.6mm);
	\draw[fill=black] (-2.5,3*\hh) circle (1.6mm);
	\draw[fill=black] (-2.5,\hh) circle (1.6mm);
	\draw[fill=black] (-2.5,-\hh) circle (1.6mm);
	\draw[fill=black] (-4,0) circle (1.6mm);
	\draw[fill=black] (-1,-2*\hh) circle (1.6mm);
	\draw[fill=black] (-1,-4*\hh) circle (1.6mm);
	\draw[fill=black] (-2.5,-3*\hh) circle (1.6mm);
	\draw[fill=black] (2,-2*\hh) circle (1.6mm);
	\draw[fill=black] (2,-4*\hh) circle (1.6mm);
	\draw[fill=black] (0.5,-3*\hh) circle (1.6mm);
	\draw[fill=black] (3.5,\hh) circle (1.6mm);
	\draw[fill=black] (3.5,-\hh) circle (1.6mm);
	\draw[fill=black] (2,0) circle (1.6mm);

\end{tikzpicture}

\\
{\red \ast2222} & {\blue 22}{\red\times} & {\blue 22}{\red\ast} & {\blue 2}{\red\ast22} & {\blue 333} 

\\ [0.25cm]

%%%%%%%%%%%%%%%%%%%%%%%%%
%      ast 333
%%%%%%%%%%%%%%%%%%%%%%%%%
\begin{tikzpicture}[scale=0.3,bend angle=45, looseness=1]

	\draw (1,0) -- (3.5,0);
	\draw (0.5,\hh) -- (60:3.5);
	\foreach \a in {0,60,120,...,300} {\draw (\a:1) -- (\a:3.5);}
	\foreach \a in {0,60,120,...,300} {\draw (\a+12:4.1) -- (\a+48:4.1);}
		
	\hexnode{(0,0)}{(0,0)}
	\foreach \a in {0,60,120,...,300} {\hexnode{(0,0)}{(\a:4.5)}}

%	\draw[fill=black] (5,0) circle (1.2mm);

\end{tikzpicture}

& 
%%%%%%%%%%%%%%%%%%%%%%%%%
%       3 ast 3
%%%%%%%%%%%%%%%%%%%%%%%%%
\begin{tikzpicture}[scale=0.5,bend angle=45, looseness=1]
	% Center hexagon
	\hex{((0,0)}{(0,0)}
	% First ring of hexagons
	\foreach \a in {30,90,...,330} {\hex{(0,0)}{(\a:\RootOfThree)}}
%	\node at (0,-3.5) {$3\!\ast\!3$};
\end{tikzpicture}

&

%%%%%%%%%%%%%%%%%%%%%%%%%
%      442
%%%%%%%%%%%%%%%%%%%%%%%%%
\begin{tikzpicture}[scale=0.5,bend angle=45, looseness=1]

	\def\h{0.9238} %height of the octagon
	\def\l{0.76537} %length of side of the octagon = side of square
	\def\c{0.5412} %length of side of the octagon = side of square

% 	\draw  (-3,-3) -- (3,-3) -- (3,3) -- (-3,3) -- cycle;
	\draw (-2.5,-1.5) -- (-1.5,-1.5) -- (-1.5,-2.5) -- (1.5,-2.5) -- (1.5,-1.5) -- (2.5,-1.5) -- (2.5,1.5) -- (1.5,1.5) -- (1.5,2.5) -- (-1.5,2.5) -- (-1.5,1.5) -- (-2.5,1.5) -- (-2.5,-1.5);\draw (-2.5,0.5) -- (0.5,0.5) -- (0.5,2.5); \draw ((-2.5,-1.5) -- (-.5,-1.5) -- (-0.5,0.5); \draw (1.5, -2.5) -- (1.5,-0.5) -- (-0.5,-0.5); \draw (1.5,1.5) -- (0.5,1.5) -- (0.5,-0.5); \draw (-0.5,-2.5) -- (-0.5,-1.5); \draw (-1.5,1.5) -- (-1.5,0.5); \draw (1.5,-1.5) -- (2.5,-1.5); \draw (1.5,-0.5) -- (2.5,-0.5);
			
	\draw[fill=black] (-2.5,-0.5) circle (1.2mm);\draw[fill=black] (-2.5,1.5) circle (1.2mm);\draw[fill=black] (-1.5,2.5) circle (1.2mm);	\draw[fill=black] (-1.5,0.5) circle (1.2mm);\draw[fill=black] (-1.5,-1.5) circle (1.2mm);\draw[fill=black] (-0.5,-2.5) circle (1.2mm);
	\draw[fill=black] (-0.5,-0.5) circle (1.2mm);\draw[fill=black] (0.5,2.5) circle (1.2mm);\draw[fill=black] (0.5,0.5) circle (1.2mm);	\draw[fill=black] (1.5,1.5) circle (1.2mm);\draw[fill=black] (1.5,-0.5) circle (1.2mm);\draw[fill=black] (1.5,-2.5) circle (1.2mm);	
	\draw[fill=black] (2.5,0.5) circle (1.2mm);\draw[fill=black] (2.5,-1.5) circle (1.2mm);	

	\draw[fill=white] (-2.5,0.5) circle (1.2mm);\draw[fill=white] (-2.5,-1.5) circle (1.2mm);\draw[fill=white] (-1.5,-2.5) circle (1.2mm);	\draw[fill=white] (-1.5,1.5) circle (1.2mm);\draw[fill=white] (-0.5,2.5) circle (1.2mm);\draw[fill=white] (-0.5,0.5) circle (1.2mm);
	\draw[fill=white] (-0.5,-1.5) circle (1.2mm);\draw[fill=white] (0.5,1.5) circle (1.2mm);\draw[fill=white] (0.5,-0.5) circle (1.2mm);		\draw[fill=white] (0.5,-2.5) circle (1.2mm);\draw[fill=white] (1.5,2.5) circle (1.2mm);\draw[fill=white] (1.5,-1.5) circle (1.2mm);	
	\draw[fill=white] (2.5,1.5) circle (1.2mm);\draw[fill=white] (2.5,-0.5) circle (1.2mm);	
	
\end{tikzpicture}

& 
%%%%%%%%%%%%%%%%%%%%%%%%%
%      ast 442
%%%%%%%%%%%%%%%%%%%%%%%%%

\begin{tikzpicture}[scale=0.5,bend angle=45, looseness=1]

	\def\h{0.9238} %height of the octagon
	\def\l{0.76537} %length of side of the octagon = side of square
	\def\c{0.5412} %length of side of the octagon = side of square
 	
		\oct{((0,0)}{(0,0)}
		\octop{((0,0)}{(22.5:1.8478)}
		\octop{((0,0)}{(-3*22.5:1.8478)}
		\oct{((0,0)}{(-22.5:1.8478+0.76539)}
		
		\Square{(0,0)}{(-45+26.565:2.9214)}
		\Square{(0,0)}{(-45-26.565:2.9214)}
		\Square{(0,0)}{(45:1.3065)}
		\Square{(0,0)}{(-135:1.3065)}		
\end{tikzpicture}

&

%%%%%%%%%%%%%%%%%%%%%%%%%
%      4 ast 2
%%%%%%%%%%%%%%%%%%%%%%%%%
\begin{tikzpicture}[scale=0.65,bend angle=45, looseness=1]

 	\draw (0.5,0) -- (0,0.5) -- (0,1) -- (0.5,1.5) -- cycle;
	\draw (0.5,1.5) -- (0.5,2) -- (0,2.5) -- (2,2.5) -- (2.5,2) -- (2.5,1.5) -- (2,1) -- (1.5,1) -- (1,1.5) -- (1,2) -- (1.5,2.5) -- (1,2) -- (0.5,2);
	\draw (1,1.5) -- (1,0) -- (1.5,0.5) -- (1.5,1) -- (1.5,0.5) -- (4,0.5) -- (4,1) -- (2,1) -- (2.5,1.5) -- (3,1.5) -- (3.5,1) -- (3.5,1);
	\draw (2,0.5) -- (2.5,0) -- (3,0) -- (3.5,0.5);
	\draw (3,1.5) -- (3,4) -- (2.5,4) -- (2.5,2) -- (2,2.5) -- (2,3) -- (2.5,3.5) -- (2,3) -- (0,3) -- (0.5,3.5) -- (1,3.5) -- (1.5,3);
	\draw (3,3.5) -- (3.5,3) -- (3.5,2.5) -- (3,2);
	\draw (0,3) -- (-0.5,3) -- (-0.5,2.5) -- (0,2.5);
	\draw (0.5,0) -- (0.5,-0.5) -- (1,-0.5) -- (1,0);

	\draw[fill=white] (-0.5,2.5) circle (0.9mm);
	\draw[fill=white] (0,1) circle (0.9mm);
	\draw[fill=white] (0,3) circle (0.9mm);
	\draw[fill=white] (0.5,0) circle (0.9mm);
	\draw[fill=white] (0.5,2) circle (0.9mm);
	\draw[fill=white] (1,-0.5) circle (0.9mm);
	\draw[fill=white] (1,1.5) circle (0.9mm);
	\draw[fill=white] (1,3.5) circle (0.9mm);
	\draw[fill=white] (1.5,0.5) circle (0.9mm);
	\draw[fill=white] (1.5,2.5) circle (0.9mm);
	\draw[fill=white] (2,1) circle (0.9mm);
	\draw[fill=white] (2,3) circle (0.9mm);
	\draw[fill=white] (2.5,0) circle (0.9mm);
	\draw[fill=white] (2.5,2) circle (0.9mm);
	\draw[fill=white] (2.5,4) circle (0.9mm);
	\draw[fill=white] (3,1.5) circle (0.9mm);
	\draw[fill=white] (3,3.5) circle (0.9mm);
	\draw[fill=white] (3.5,0.5) circle (0.9mm);
	\draw[fill=white] (3.5,2.5) circle (0.9mm);
	\draw[fill=white] (4,1) circle (0.9mm);

	\draw[fill=black] (-0.5,3) circle (0.9mm);
	\draw[fill=black] (0,0.5) circle (0.9mm);
	\draw[fill=black] (0,2.5) circle (0.9mm);
	\draw[fill=black] (0.5,1.5) circle (0.9mm);
	\draw[fill=black] (0.5,3.5) circle (0.9mm);
	\draw[fill=black] (0.5,-0.5) circle (0.9mm);
	\draw[fill=black] (1,0) circle (0.9mm);
	\draw[fill=black] (1,2) circle (0.9mm);
	\draw[fill=black] (1.5,1) circle (0.9mm);
	\draw[fill=black] (1.5,3) circle (0.9mm);
	\draw[fill=black] (2,0.5) circle (0.9mm);
	\draw[fill=black] (2,2.5) circle (0.9mm);
	\draw[fill=black] (2.5,1.5) circle (0.9mm);
	\draw[fill=black] (2.5,3.5) circle (0.9mm);
	\draw[fill=black] (3,0) circle (0.9mm);
	\draw[fill=black] (3,2) circle (0.9mm);
	\draw[fill=black] (3,4) circle (0.9mm);
	\draw[fill=black] (3.5,1) circle (0.9mm);
	\draw[fill=black] (3.5,3) circle (0.9mm);
	\draw[fill=black] (4,0.5) circle (0.9mm);

\end{tikzpicture}

\\

{\red \ast333} & {\blue 3}{\red\ast3}  & {\blue 442} & {\red\ast442} & {\blue4}{\red\ast2} 
\\ [0.25cm]

%%%%%%%%%%%%%%%%%%%%%%%%%
%     632
%%%%%%%%%%%%%%%%%%%%%%%%%

\begin{tikzpicture}[scale=0.45,bend angle=45, looseness=1]

% square lines	
	\foreach \b in {0,60,120,180,240,300} {\draw (60+\b:1) -- (50.85+\b:1.42);}
	\foreach \b in {0,60,120,180,240,300} {\draw (40+\b:0.9) -- (36.8+\b:1.35);}	
	\foreach \b in {0,60,120,180,240,300} {\draw (10+\b:1.77) -- (25+\b:1.925);}
	\foreach \b in {0,60,120,180,240,300} {\draw (7.9+\b:2.1) -- (21.5+\b:2.235);}

% hexagons
	\hexap{(0,0)}{(0:0)}
	\foreach \a in {30,90,150,210,270,330} {\hexap{(0,0)}{(\a+16.5:2.3)}}
	
\end{tikzpicture}

& 
%%%%%%%%%%%%%%%%%%%%%%%%%
%      ast 632
%%%%%%%%%%%%%%%%%%%%%%%%%
\begin{tikzpicture}[scale=0.45,bend angle=45, looseness=1]
	\def\hhex{0.64952}
	\def\rdode{2-\hhex+0.065}
	
% sides of squares (1st)
	\draw (45:\rdode) -- (75:\rdode);	\draw (105:\rdode) -- (135:\rdode);\draw (165:\rdode) -- (195:\rdode);
	\draw (225:\rdode) -- (255:\rdode);\draw (285:\rdode) -- (315:\rdode);\draw (345:\rdode) -- (15:\rdode);
% sides of squares (2nd)	
	\draw (50:2.15) -- (70:2.15);\draw (110:2.15) -- (130:2.15);\draw (170:2.15) -- (190:2.15);
	\draw (230:2.15) -- (250:2.15);\draw (290:2.15) -- (310:2.15);\draw (350:2.15) -- (10:2.15);
	
% exterior edges
	% up-right
	\draw (2+\hhex+0.45,2+\hhex) -- (2+\hhex+0.08,2) -- (2+\hhex+0.45,2-\hhex);
	\draw (38:2.67) -- (36.5:3.375); \draw (22:2.67) -- (23.5:3.375); 
	% down-right
	\draw (2+\hhex+0.45,-2-\hhex) -- (2+\hhex+0.08,-2) -- (2+\hhex+0.45,-2+\hhex);
	\draw (-38:2.67) -- (-36.5:3.375); \draw (-22:2.67) -- (-23.5:3.375); 
	% up-left
	\draw (-2-\hhex-0.45,2+\hhex) -- (-2-\hhex-0.08,2) -- (-2-\hhex-0.45,2-\hhex);
	\draw (38+120:2.67) -- (36.5+120:3.375); \draw (22+120:2.67) -- (23.5+120:3.375); 
	% down-left
	\draw (-2-\hhex-0.45,-2-\hhex) -- (-2-\hhex-0.08,-2) -- (-2-\hhex-0.45,-2+\hhex);
	\draw (-38-120:2.67) -- (-36.5-120:3.375); \draw (-22-120:2.67) -- (-23.5-120:3.375); 
	
% outside hexagons
	\foreach \a in {30,150,270} {\hexp{(0,0)}{(\a:2)}}
	\foreach \a in {90,210,330} {\hexpop{(0,0)}{(\a:2)}}

% exterior nodes 
	%(up-right)
	\draw[fill=white] (2+\hhex+0.45,2+\hhex) circle (1.2mm);
	\draw[fill=white] (2+\hhex+0.45,2-\hhex) circle (1.2mm);
	\draw[fill=black] (2+\hhex+0.08,2) circle (1.2mm);
	%(down-right)	
	\draw[fill=black] (2+\hhex+0.45,-2-\hhex) circle (1.2mm);
	\draw[fill=black] (2+\hhex+0.45,-2+\hhex) circle (1.2mm);
	\draw[fill=white] (2+\hhex+0.08,-2) circle (1.2mm);
	% exterior nodes (up-left)
	\draw[fill=black] (-2-\hhex-0.45,2+\hhex) circle (1.2mm);
	\draw[fill=black] (-2-\hhex-0.45,2-\hhex) circle (1.2mm);
	\draw[fill=white] (-2-\hhex-0.08,2) circle (1.2mm);
	% exterior nodes (down-left)
	\draw[fill=white] (-2-\hhex-0.45,-2-\hhex) circle (1.2mm);
	\draw[fill=white] (-2-\hhex-0.45,-2+\hhex) circle (1.2mm);
	\draw[fill=black] (-2-\hhex-0.08,-2) circle (1.2mm);
	
\end{tikzpicture}
&&
\\
{\blue632}  & {\red\ast632} & &
\end{array}
\]
\caption{Examples of consistent dimer models for the 17 plane symmetry types}
\label{fg:dimer17}
\end{figure}

%\begin{acknowledgements}
%If you'd like to thank anyone, place your comments here
%and remove the percent signs.
%\end{acknowledgements}

% Authors must disclose all relationships or interests that 
% could have direct or potential influence or impart bias on 
% the work: 
%
% \section*{Conflict of interest}
%
% The authors declare that they have no conflict of interest.

% BibTeX users please use one of
%\bibliographystyle{spbasic}      % basic style, author-year citations
\bibliographystyle{amsalpha}      % mathematics and physical sciences
\bibliography{bibs}   % name your BibTeX data base

\newcommand{\etalchar}[1]{$^{#1}$}
\def\cprime{$'$} \def\cprime{$'$}
\providecommand{\bysame}{\leavevmode\hbox to3em{\hrulefill}\thinspace}
\providecommand{\MR}{\relax\ifhmode\unskip\space\fi MR }
% \MRhref is called by the amsart/book/proc definition of \MR.
\providecommand{\MRhref}[2]{%
  \href{http://www.ams.org/mathscinet-getitem?mr=#1}{#2}
}
\providecommand{\href}[2]{#2}
\begin{thebibliography}{FHM{\etalchar{+}}06}

\bibitem[Bax89]{Baxter_ESMSM}
Rodney~J. Baxter, \emph{Exactly solved models in statistical mechanics}, Academic Press Inc. [Harcourt Brace Jovanovich Publishers], London, 1989, Reprint of the 1982 original. \MR{998375 (90b:82001)}

\bibitem[BCQV15]{MR3319545}
Raf Bocklandt, Alastair Craw, and Alexander Quintero~V\'{e}lez, \emph{Geometric {R}eid's recipe for dimer models}, Math. Ann. \textbf{361} (2015), no.~3-4, 689--723. \MR{3319545}

\bibitem[BIU]{Beil-Ishii-Ueda_CDM}
Charlie Beil, Akira Ishii, and Kazushi Ueda, \emph{Cancellativization of dimer models}, arXiv:1301.5410.

\bibitem[Boc12]{Bocklandt_CCDM}
Raf Bocklandt, \emph{Consistency conditions for dimer models}, Glasg. Math. J. \textbf{54} (2012), no.~2, 429--447. \MR{2911380}

\bibitem[Boc13]{MR3010162}
\bysame, \emph{Calabi--{Y}au algebras and weighted quiver polyhedra}, Math. Z. \textbf{273} (2013), no.~1-2, 311--329. \MR{3010162}

\bibitem[Boc16]{MR3509904}
\bysame, \emph{A dimer {ABC}}, Bull. Lond. Math. Soc. \textbf{48} (2016), no.~3, 387--451. \MR{3509904}

\bibitem[Bri02]{MR1893007}
Tom Bridgeland, \emph{Flops and derived categories}, Invent. Math. \textbf{147} (2002), no.~3, 613--632. \MR{MR1893007 (2003h:14027)}

\bibitem[Bro12]{Broomhead}
Nathan Broomhead, \emph{Dimer models and {C}alabi-{Y}au algebras}, Mem. Amer. Math. Soc. \textbf{215} (2012), no.~1011, viii+86. \MR{2908565}

\bibitem[BZ05]{Butti-Zaffaroni_RTD}
Agostino Butti and Alberto Zaffaroni, \emph{{$R$}-charges from toric diagrams and the equivalence of {$a$}-maximization and {$Z$}-minimization}, J. High Energy Phys. (2005), no.~11, 019, 42 pp. (electronic). \MR{MR2187558 (2006i:81154)}

\bibitem[BZ06]{Butti-Zaffaroni_TGQGT}
\bysame, \emph{From toric geometry to quiver gauge theory: the equivalence of {$a$}-maximization and {$Z$}-minimization}, Fortschr. Phys. \textbf{54} (2006), no.~5-6, 309--316. \MR{MR2230600 (2007i:81185)}

\bibitem[CBGS08]{Conway}
John Conway, Heidi Burgiel, and Chaim Goodman-Strauss, \emph{The symmetries of things}, second ed., Graduate Studies in Mathematics, vol.~19, A K Peters/CRC Press, 2008.

\bibitem[CI04]{MR2078369}
Alastair Craw and Akira Ishii, \emph{Flops of {$G$}-{H}ilb and equivalences of derived categories by variation of {GIT} quotient}, Duke Math. J. \textbf{124} (2004), no.~2, 259--307. \MR{MR2078369}

\bibitem[Dav11]{Davison}
Ben Davison, \emph{Consistency conditions for brane tilings}, J. Algebra \textbf{338} (2011), 1--23. \MR{2805177 (2012e:14110)}

\bibitem[FHKV08]{Feng-He-Kennaway-Vafa}
Bo~Feng, Yang-Hui He, Kristian~D. Kennaway, and Cumrun Vafa, \emph{Dimer models from mirror symmetry and quivering amoebae}, Adv. Theor. Math. Phys. \textbf{12} (2008), no.~3, 489--545. \MR{MR2399318 (2009k:81180)}

\bibitem[FHM{\etalchar{+}}06]{Franco-Hanany-Martelli-Sparks-Vegh-Wecht_GTTGBT}
Sebasti{\'a}n Franco, Amihay Hanany, Dario Martelli, James Sparks, David Vegh, and Brian Wecht, \emph{Gauge theories from toric geometry and brane tilings}, J. High Energy Phys. (2006), no.~1, 128, 40 pp. (electronic). \MR{MR2201204}

\bibitem[FHV{\etalchar{+}}06]{Franco-Hanany-Vegh-Wecht-Kennaway_BDQGT}
Sebasti{\'a}n Franco, Amihay Hanany, David Vegh, Brian Wecht, and Kristian~D. Kennaway, \emph{Brane dimers and quiver gauge theories}, J. High Energy Phys. (2006), no.~1, 096, 48 pp. (electronic). \MR{MR2201227}

\bibitem[FR37]{Fowler-Rushbrooke}
R.~H. Fowler and G.~S. Rushbrooke, \emph{An attempt to extend the statistical theory of perfect solutions}, Trans. Faraday Soc. \textbf{33} (1937), 1272 -- 1294.

\bibitem[FU10]{Futaki-Ueda_A-infinity}
Masahiro Futaki and Kazushi Ueda, \emph{Exact {L}efschetz fibrations associated with dimer models}, Math. Res. Lett. \textbf{17} (2010), no.~6, 1029--1040. \MR{2729627}

\bibitem[FV06]{Franco-Vegh_MSGTDM}
Sebasti{\'a}n Franco and David Vegh, \emph{Moduli spaces of gauge theories from dimer models: proof of the correspondence}, J. High Energy Phys. (2006), no.~11, 054, 26 pp. (electronic). \MR{MR2270405 (2007j:81161)}

\bibitem[Gul08]{Gulotta}
Daniel~R. Gulotta, \emph{Properly ordered dimers, {$R$}-charges, and an efficient inverse algorithm}, J. High Energy Phys. (2008), no.~10, 014, 31. \MR{MR2453031 (2010b:81116)}

\bibitem[HHV06]{Hanany-Herzog-Vegh_BTEC}
Amihay Hanany, Christopher~P. Herzog, and David Vegh, \emph{Brane tilings and exceptional collections}, J. High Energy Phys. (2006), no.~7, 001, 44 pp. (electronic). \MR{MR2240899 (2008b:81224)}

\bibitem[HK05]{Hanany-Kennaway_DMTD}
Amihay Hanany and Kristian~D. Kennaway, \emph{Dimer models and toric diagrams}, hep-th/0503149, 2005.

\bibitem[HV07]{Hanany-Vegh}
Amihay Hanany and David Vegh, \emph{Quivers, tilings, branes and rhombi}, J. High Energy Phys. (2007), no.~10, 029, 35. \MR{MR2357949}

\bibitem[IU]{Ishii-Ueda_DMEC}
Akira Ishii and Kazushi Ueda, \emph{Dimer models and exceptional collections}, arXiv:0911.4529.

\bibitem[IU08]{Ishii-Ueda_08}
\bysame, \emph{On moduli spaces of quiver representations associated with dimer models}, Higher dimensional algebraic varieties and vector bundles, RIMS K\^oky\^uroku Bessatsu, B9, Res. Inst. Math. Sci. (RIMS), Kyoto, 2008, pp.~127--141. \MR{MR2509696}

\bibitem[IU11]{Ishii-Ueda_CCDM}
\bysame, \emph{A note on consistency conditions on dimer models}, Higher dimensional algebraic varieties, RIMS K\^oky\^uroku Bessatsu, B24, Res. Inst. Math. Sci. (RIMS), Kyoto, 2011, pp.~143--164.

\bibitem[IU15]{Ishii-Ueda_DMSMC}
\bysame, \emph{Dimer models and the special {McKay} correspondence}, Geom. Topol. \textbf{19} (2015), no.~6, 3405--3466. \MR{3447107}

\bibitem[IU16]{MR3532120}
\bysame, \emph{Dimer models and crepant resolutions}, Hokkaido Math. J. \textbf{45} (2016), no.~1, 1--42. \MR{3532120}

\bibitem[Kat07]{Kato_ZFSFT}
Akishi Kato, \emph{Zonotopes and four-dimensional superconformal field theories}, J. High Energy Phys. (2007), no.~6, 037, 30 pp. (electronic). \MR{MR2326614 (2009a:81201)}

\bibitem[Ken04]{Kenyon_IDM}
Richard Kenyon, \emph{An introduction to the dimer model}, School and {C}onference on {P}robability {T}heory, ICTP Lect. Notes, XVII, Abdus Salam Int. Cent. Theoret. Phys., Trieste, 2004, pp.~267--304 (electronic). \MR{MR2198850 (2006k:82033)}

\bibitem[MR10]{Mozgovoy-Reineke}
Sergey Mozgovoy and Markus Reineke, \emph{On the noncommutative {D}onaldson-{T}homas invariants arising from brane tilings}, Adv. Math. \textbf{223} (2010), no.~5, 1521--1544. \MR{2592501}

\bibitem[MSY06]{Martelli-Sparks-Yau_GDAM}
Dario Martelli, James Sparks, and Shing-Tung Yau, \emph{The geometric dual of {$a$}-maximisation for toric {S}asaki-{E}instein manifolds}, Comm. Math. Phys. \textbf{268} (2006), no.~1, 39--65. \MR{MR2249795 (2008c:53037)}

\bibitem[NdC12]{NdC12}
A.~Nolla~de Celis, \emph{$g$-graphs and special representations for binary dihedral groups in $gl(2,\mathbb{C})$}, Glasgow Math. J. \textbf{55} (2012), no.~1, 23--57.

\bibitem[Sze08]{Szendroi_NCDT}
Bal{\'a}zs Szendr{\H{o}}i, \emph{Non-commutative {D}onaldson-{T}homas invariants and the conifold}, Geom. Topol. \textbf{12} (2008), no.~2, 1171--1202. \MR{MR2403807 (2009e:14100)}

\bibitem[TU10]{Toda-Uehara}
Yukinobu Toda and Hokuto Uehara, \emph{Tilting generators via ample line bundles}, Adv. Math. \textbf{223} (2010), no.~1, 1--29. \MR{2563209}

\bibitem[UY11]{Ueda-Yamazaki_NBTMQ}
Kazushi Ueda and Masahito Yamazaki, \emph{A note on dimer models and {M}c{K}ay quivers}, Comm. Math. Phys. \textbf{301} (2011), no.~3, 723--747. \MR{2784278}

\bibitem[UY13]{Ueda-Yamazaki_toricdP}
\bysame, \emph{Homological mirror symmetry for toric orbifolds of toric del {P}ezzo surfaces}, J. Reine Angew. Math. \textbf{680} (2013), 1--22. \MR{3100950}

\bibitem[vdB04a]{MR2077594}
Michel van~den Bergh, \emph{Non-commutative crepant resolutions}, The legacy of {N}iels {H}enrik {A}bel, Springer, Berlin, 2004, pp.~749--770. \MR{2077594 (2005e:14002)}

\bibitem[VdB04b]{MR2057015}
Michel Van~den Bergh, \emph{Three-dimensional flops and noncommutative rings}, Duke Math. J. \textbf{122} (2004), no.~3, 423--455. \MR{MR2057015 (2005e:14023)}

\bibitem[Wun87]{Wunram2}
J.~Wunram, \emph{Reflexive modules on cyclic quotient surface singularities}, Singularities, representation of algebras, and vector bundles ({L}ambrecht, 1985), Lecture Notes in Math., vol. 1273, Springer, Berlin, 1987, pp.~221--231. \MR{MR915177 (88m:14023)}

\bibitem[Wun88]{Wunram}
J{\"u}rgen Wunram, \emph{Reflexive modules on quotient surface singularities}, Math. Ann. \textbf{279} (1988), no.~4, 583--598. \MR{MR926422 (89g:14029)}

\bibitem[Yam08]{MR2423955}
Masahito Yamazaki, \emph{Brane tilings and their applications}, Fortschr. Phys. \textbf{56} (2008), no.~6, 555--686. \MR{2423955 (2009e:81205)}

\bibitem[YY93]{YY93}
Stephen S.-T. Yau and Yung Yu, \emph{Gorenstein quotient singularities in dimension three}, Mem. Amer. Math. Soc. \textbf{105} (1993), no.~505, viii+88. \MR{1169227 (94b:14045)}

\end{thebibliography}

% Non-BibTeX users please use
%\begin{thebibliography}{}
%
% and use \bibitem to create references. Consult the Instructions
% for authors for reference list style.
%
%\bibitem{RefJ}
% Format for Journal Reference
%Author, Article title, Journal, Volume, page numbers (year)
% Format for books
%\bibitem{RefB}
%Author, Book title, page numbers. Publisher, place (year)
% etc
%\end{thebibliography}

\noindent
Akira Ishii

Graduate School of Mathematics,
Nagoya University,
Furocho, Chikusa-ku, Nagoya, 464-8602,
Japan

{\em e-mail address}\ : \ akira141@math.nagoya-u.ac.jp

\ \\

\noindent
\'Alvaro Nolla de Celis

Universidad Aut\'onoma de Madrid (UAM),
Faculty of Teachers Training and Education,
Tom\'as y Valiente 3,
28049 Madrid,
Spain.

{\em e-mail address}\ : \ alvaro.nolla@uam.es

\ \\

\noindent
Kazushi Ueda

Graduate School of Mathematical Sciences,
The University of Tokyo,
3-8-1,
Meguro-ku,
Tokyo,
153-8914,
Japan.

{\em e-mail address}\ : \  kazushi@ms.u-tokyo.ac.jp

\end{document}